\documentclass{amsart}
\usepackage{amssymb}
\usepackage{amsfonts}
\usepackage[all]{xy}
\usepackage{graphicx}

\def\proprigid{Proposition~4.11}
\def\thmbaricglue{Theorem~4.12}
\def\sectbariccoh{Section~6}
\def\lemqext{Lemma~6.3}
\def\thmbaric{Theorem~6.4}
\def\lembaricres{Lemma~6.6}
\def\thmstagcoh{Theorem~8.1}
\def\thmstagdual{Theorem~8.6}

\newcommand{\incgr}[1]{\hbox{$\vcenter{%
  \hbox{\includegraphics[height=24pt]%
  {#1.eps}}}$}}

\newcommand{\sst}{\scriptscriptstyle}

\newcommand{\C}{\mathbb{C}}

\newcommand{\fD}{\mathfrak{D}}
\newcommand{\D}{\mathbb{D}}
\newcommand{\F}{\mathbb{F}}
\newcommand{\N}{\mathbb{N}}
\newcommand{\Q}{\mathbb{Q}}
\newcommand{\Z}{\mathbb{Z}}

\newcommand{\bru}{{}^{\sst \bar r}}
\newcommand{\ru}{{}^{\sst r}\!}

\newcommand{\trl}{{}_{\sst 2r}}
\newcommand{\tql}{{}_{\sst 2q}}
\newcommand{\qprime}{q'}

\newcommand{\dualp}{\tilde p}
\newcommand{\dualpforq}{\tilde q}
\newcommand{\dualq}{\hat q}
\newcommand{\dualqprime}{\hat q'}

\newcommand{\dualr}{\bar r}

\newcommand{\ufl}{\mathchoice{{}^\flat}{{}^\flat}{{}^\flat}{\flat}}
\newcommand{\ush}{\mathchoice{{}^\sharp}{{}^\sharp}{{}^\sharp}{\sharp}}
\newcommand{\flatru}{{}^{\sst \flat r}}
\newcommand{\sharpru}{{}^{\sst \sharp r}}

\newcommand{\q}{{}_{\sst q}}
\newcommand{\qu}{{}^{\sst q}}
\newcommand{\pq}{{}^{\sst p}_{\sst q}}
\newcommand{\bpcq}{{}^{\sst \dualp}_{\sst \dualq}}

\newcommand{\barq}{{}^{\sst \bar q}}
\newcommand{\cheq}{{}_{\sst \dualq}}

\newcommand{\sdualq}{\breve q}
\newcommand{\breq}{{}_{\sst \sdualq}}
\newcommand{\tbql}{{}_{\sst 2\sdualq}}
\newcommand{\skewr}{\llcorner r \lrcorner}
\newcommand{\skewdr}{\llcorner \dualr \lrcorner}
\newcommand{\lsr}{{}_{\sst \skewr}}
\newcommand{\tsr}{{}_{\sst 2\skewr}}
\newcommand{\lsdr}{{}_{\sst \skewdr}}

\newcommand{\flatpq}{{}^{\sst\ufl p}_{\sst \hphantom{\flat}q}}
\newcommand{\sharppq}{{}^{\sst\ush p}_{\sst \hphantom{\sharp}q}}

\newcommand{\tj}{\tilde\jmath}

\newcommand{\orb}[1]{\mathbb{O}(#1)}

\newcommand{\cC}{\mathcal{C}}
\newcommand{\cD}{\mathcal{D}}
\newcommand{\cF}{\mathcal{F}}
\newcommand{\cG}{\mathcal{G}}
\newcommand{\cH}{\mathcal{H}}
\newcommand{\cI}{\mathcal{I}}

\newcommand{\cL}{\mathcal{L}}
\newcommand{\cM}{\mathcal{M}}
\newcommand{\cO}{\mathcal{O}}
\newcommand{\cP}{\mathcal{P}}
\newcommand{\cQ}{\mathcal{Q}}
\newcommand{\cR}{\mathcal{R}}
\newcommand{\cS}{\mathcal{S}}

\newcommand{\bbA}{\mathbb{A}}
\newcommand{\bbG}{\mathbb{G}}

\newcommand{\cIC}{\mathcal{IC}}

\newcommand{\Pp}[2]{\cP(#1)_{\sst [#2]}}
\newcommand{\Pnp}[2]{\cP^\natural(#1)_{\sst [#2]}}

\newcommand{\csupp}[2]{\cC^{\mathrm{supp}}_G(#1,#2)}

\newcommand{\qsupp}[2]{\cQ^{\mathrm{supp}}_G(#1,#2)}

\newcommand{\cg}[1]{\cC_G(#1)}
\newcommand{\cgl}[2]{\cC_G(#1)_{\le #2}}
\newcommand{\cgg}[2]{\cC_G(#1)_{\ge #2}}

\newcommand{\qg}[1]{\cQ_G(#1)}

\newcommand{\Db}[1]{\cD_{\sst G}^{\sst\mathrm{b}}(#1)}
\newcommand{\Dm}[1]{\cD_{\sst G}^{\sst -}(#1)}
\newcommand{\Dp}[1]{\cD_{\sst G}^{\sst +}(#1)}

\newcommand{\Dl}[2]{\Db{#1}^{\sst\le #2}}
\newcommand{\Dlt}[2]{\Db{#1}^{\sst< #2}}
\newcommand{\Dg}[2]{\Db{#1}^{\sst\ge #2}}
\newcommand{\Dgt}[2]{\Db{#1}^{\sst> #2}}
\newcommand{\Dml}[2]{\Dm{#1}^{\sst\le #2}}
\newcommand{\Dmlt}[2]{\Dm{#1}^{\sst< #2}}
\newcommand{\Dpg}[2]{\Dp{#1}^{\sst\ge #2}}
\newcommand{\Dpgt}[2]{\Dp{#1}^{\sst> #2}}

\newcommand{\Dsl}[2]{\Db{#1}_{\sst\le #2}}
\newcommand{\Dslt}[2]{\Db{#1}_{\sst< #2}}
\newcommand{\Dsg}[2]{\Db{#1}_{\sst\ge #2}}

\newcommand{\Dsp}[2]{\Db{#1}_{\sst [#2]}}
\newcommand{\Dmsl}[2]{\Dm{#1}_{\sst\le #2}}
\newcommand{\Dmslt}[2]{\Dm{#1}_{\sst< #2}}
\newcommand{\Dpsg}[2]{\Dp{#1}_{\sst\ge #2}}
\newcommand{\Dpsgt}[2]{\Dp{#1}_{\sst> #2}}

\newcommand{\Dll}[3]{\Dl{#1}{#2}_{\sst\le #3}}

\newcommand{\Dlp}[3]{\Dl{#1}{#2}_{\sst [#3]}}
\newcommand{\Dltp}[3]{\Dlt{#1}{#2}_{\sst [#3]}}
\newcommand{\Dgg}[3]{\Dg{#1}{#2}_{\sst\ge #3}}

\newcommand{\Dgp}[3]{\Dg{#1}{#2}_{\sst [#3]}}
\newcommand{\Dgtp}[3]{\Dgt{#1}{#2}_{\sst [#3]}}
\newcommand{\Dmll}[3]{\Dml{#1}{#2}_{\sst\le #3}}
\newcommand{\Dmltl}[3]{\Dmlt{#1}{#2}_{\sst\le #3}}

\newcommand{\Dpgg}[3]{\Dpg{#1}{#2}_{\sst\ge #3}}
\newcommand{\Dpgtg}[3]{\Dpgt{#1}{#2}_{\sst\ge #3}}

\newcommand{\Dkl}[2]{\Db{#1}_{\sst\sqsubseteq #2}}
\newcommand{\Dkg}[2]{\Db{#1}_{\sst\sqsupseteq #2}}
\newcommand{\Dmkl}[2]{\Dm{#1}_{\sst\sqsubseteq #2}}
\newcommand{\Dpkg}[2]{\Dp{#1}_{\sst\sqsupseteq #2}}
\newcommand{\Dkp}[2]{\Db{#1}_{\sst\langle #2\rangle}}

\newcommand{\ttrunc}{\tau}
\newcommand{\Tl}[1]{\ttrunc^{\sst\le #1}}
\newcommand{\Tlt}[1]{\ttrunc^{\sst< #1}}
\newcommand{\Tg}[1]{\ttrunc^{\sst\ge #1}}
\newcommand{\Tgt}[1]{\ttrunc^{\sst> #1}}

\newcommand{\Tlp}[2]{\Tl{#1}_{\sst [#2]}}
\newcommand{\Tltp}[2]{\Tlt{#1}_{\sst [#2]}}
\newcommand{\Tgp}[2]{\Tg{#1}_{\sst [#2]}}
\newcommand{\Tgtp}[2]{\Tgt{#1}_{\sst [#2]}}

\newcommand{\strunc}{\sigma}
\newcommand{\Sl}[1]{\strunc_{\sst\le #1}}
\newcommand{\Sg}[1]{\strunc_{\sst\ge #1}}

\newcommand{\btrunc}{\beta}
\newcommand{\Bl}[1]{\btrunc_{\sst\le #1}}
\newcommand{\Blt}[1]{\btrunc_{\sst< #1}}
\newcommand{\Bg}[1]{\btrunc_{\sst\ge #1}}
\newcommand{\Bgt}[1]{\btrunc_{\sst> #1}}

\newcommand{\half}{\mathchoice{{\textstyle\frac{1}{2}}}{\frac{1}{2}}{\frac{1}{2}}{\frac{1}{2}}}

\newcommand{\rg}[1]{\cR(#1)}
\newcommand{\rgl}[2]{\cR(#1)_{\sst \le #2}}
\newcommand{\rgg}[2]{\cR(#1)_{\sst \ge #2}}

\newcommand{\fu}{\mathfrak{u}}

\newcommand{\dbmc}[1]{\mathrm{D}^{\sst\mathrm{b}}_{\sst\mathrm{m}}(#1)}

\newcommand{\red}{{\mathrm{red}}}

\newcommand{\hto}{\hookrightarrow}
\newcommand{\ssm}{\smallsetminus}

\DeclareMathOperator{\im}{im}
\DeclareMathOperator{\cok}{cok}
\DeclareMathOperator{\alt}{alt}

\DeclareMathOperator{\cod}{cod}
\DeclareMathOperator{\scod}{scod}

\DeclareMathOperator{\step}{step}
\DeclareMathOperator{\Hom}{Hom}
\DeclareMathOperator{\Ext}{Ext}

\DeclareMathOperator{\cRHom}{\mathit{R}\mathcal{H}\mathit{om}}
\newcommand{\Lotimes}{\mathchoice%
  {\overset{\scriptscriptstyle L}{\otimes}}%
  {\otimes^{\scriptscriptstyle L}}{\otimes^L}{\otimes^L}}

\DeclareMathOperator{\Spec}{Spec}

\newtheorem{thm}{Theorem}[section]
\newtheorem{lem}[thm]{Lemma}
\newtheorem{prop}[thm]{Proposition}
\newtheorem{cor}[thm]{Corollary}

\theoremstyle{definition}
\newtheorem{defn}[thm]{Definition}

\theoremstyle{remark}
\newtheorem{rmk}[thm]{Remark}

\numberwithin{equation}{section}

\title[Purity and decomposition theorems]{Purity and decomposition
theorems for staggered sheaves}
\author{Pramod N.~Achar}
\thanks{The first author was partially supported by NSF Grant DMS-0500873.}
\author{David Treumann}
\date{August 23, 2008}

\begin{document}

\begin{abstract}
Two major results in the theory of $\ell$-adic mixed constructible sheaves are the purity theorem (every simple perverse sheaf is pure) and the decomposition theorem (every pure object in the derived category is a direct sum of shifts of simple perverse sheaves).  In this paper, we prove analogues of these results for coherent sheaves.  Specificially, we work with {\it staggered sheaves}, which form the heart of a certain $t$-structure on the derived category of equivariant coherent sheaves.  We prove, under some reasonable hypotheses, that every simple staggered sheaf is pure, and that every pure complex of coherent sheaves is a direct sum of shifts of simple staggered sheaves.
\end{abstract}

\maketitle

\section{Introduction}
\label{sect:intro}

Let $Z$ be a variety over a finite field $\F_q$, and let $\dbmc Z$ denote the bounded derived category of $\ell$-adic mixed constructible sheaves on $Z$.  Recall that the \emph{weights} of an object $F \in \dbmc Z$ are certain integers defined in terms of the eigenvalues of the Frobenius morphism on the stalks at $F$ at $\F_q$-points of $Z$.  An object is said to be \emph{pure} of weight $w \in \Z$ if both it and its Verdier dual have weights $\le w$.  The theory of weights and purity plays a vital role in the proof and in applications of the Weil conjectures~\cite{del1,del2,bbd}.  

Two of the most astonishing consequences of the Weil conjectures occur in the theory of perverse sheaves, developed in~\cite[Chap. 5]{bbd}.  They are (i)~the Purity Theorem~\cite[Th\'eor\`eme~5.3.5]{bbd}, which states that every perverse sheaf has a canonical filtration with pure subquotients (and in particular that every simple perverse sheaf is pure), and (ii)~the Decomposition Theorem~\cite[Th\'eor\`emes 5.3.8 and 5.4.5]{bbd}, which states that every pure object in $\dbmc Z$ is a direct sum of shifts of simple perverse sheaves.  (A more familiar statement of the decomposition theorem---that the pushforward of a pure perverse sheaf along a proper morphism admits such a decomposition---is a consequence of (ii) and Deligne's reformulation of the Weil conjectures~\cite[Th\'eor\`eme I]{del2}).  These two theorems are the source of much of the power of the theory of perverse sheaves for applications in representation theory and other areas.

In this paper, we seek analogues of these results in the setting of derived categories of equivariant coherent sheaves.  Let $X$ be a scheme of finite type over an arbitrary field, and let $G$ be an affine algebraic group acting on $X$ with finitely many orbits.  Let $\Db X$ denote the bounded derived category of $G$-equivariant coherent sheaves on $X$.  The category of \emph{staggered sheaves}, introduced in~\cite{a}, is the heart of a certain nonstandard $t$-structure on $\Db X$.  This category shares some of the key properties of perverse sheaves: for example, every object has finite length, and the simple objects arise via the ``$\cIC$ functor'' from irreducible vector bundles on orbits.

In $\Db X$ there is no single best notion of weight or purity as there is in the $\ell$-adic setting.  Rather, there is a large number of such notions parameterized by \emph{baric perversities}, which are certain integer-valued functions on the set of $G$-orbits in $X$.  More precisely, in~\cite{at:bs} we associated to each baric perversity a \emph{baric structure} (a certain kind of filtration of a triangulated category) on $\Db X$, which we use here to simulate the formalism of weights.  We call an object $\cF \in \Db X$ \emph{pure} of baric degree $w$ if both it and its Serre--Grothendieck dual lie in the ``$\le w$'' part of the baric structure.  (A result of S.~Morel~\cite{mor} essentially states that
Frobenius weights give rise to a baric structure on $\dbmc Z$, so the
theory of $\ell$-adic mixed perverse sheaves could be redeveloped using the
language of baric structures as well.)

The main results of the present paper (which are Theorems \ref{thm:purity}, \ref{thm:purity2}, and \ref{thm:decomposition}) come in two incarnations, a ``baric''
one and a ``skew'' one.  In the baric version, they state that under
some reasonable hypotheses, every staggered sheaf has a canonical
filtration with pure subquotients, and every pure object of $\Db X$ is a
direct sum of shifts of simple staggered sheaves.  The skew versions consist of essentially the same statements, but with
purity replaced by a new concept called \emph{skew-purity}.

An outline of the paper is as follows.  We begin in
Section~\ref{sect:prelim} by fixing notation and recalling relevant results
about baric structures and staggered sheaves.  In Sections~\ref{sect:pure}
and~\ref{sect:purep}, we construct two $t$-structures on the full
triangulated subcategory of pure objects of baric degree $w$ in $\Db X$,
called the \emph{purified standard $t$-structure} and the
\emph{pure-perverse $t$-structure}.  (The latter is defined in terms of the
former.)  We also prove that the heart of the pure-perverse $t$-structure
is contained in that of the staggered $t$-structure.  In
Section~\ref{sect:ic}, we study simple objects in the pure-perverse
$t$-structure.  They, like simple staggered sheaves, are characterized by a
certain uniqueness property, and this allows us to prove that every simple
staggered sheaf lies in the heart of a suitable pure-perverse
$t$-structure.  This is a major step towards the baric version of the
Purity Theorem, whose proof is completed in Section~\ref{sect:purity}.

Next, in Section~\ref{sect:orbit}, which is essentially independent of the
rest of the paper, we give a combinatorial classification of staggered
$t$-structures on a variety consisting of a single $G$-orbit.  This allows
us to give an elementary criterion for a certain $\Ext^1$-vanishing
condition that appears as a hypothesis throughout the rest of the paper. 
Section~\ref{sect:ext-van} contains some results on vanishing of higher
$\Ext$-groups; these lay the the groundwork for the definition of
skew-purity in Section~\ref{sect:skew}.  The skew version of the Purity
Theorem is proved in Section~\ref{sect:purity2}, and both versions of the
Decomposition Theorem are proved together in Section~\ref{sect:decomp}.  Finally, Section~\ref{sect:example} gives a brief example.

\section{Preliminaries and Notation}
\label{sect:prelim}

Let $X$ be a scheme of finite type over a field $\Bbbk$.  Let $G$ be an affine algebraic group over $\Bbbk$, acting on $X$.  Assume that $G$ acts on $X$ with finitely many orbits.  Here, and throughout the paper, an \emph{orbit} is a reduced, locally closed $G$-invariant subscheme containing no proper nonempty closed $G$-invariant subschemes.  $X$ itself need not be reduced.  Let $\orb X$ denote the set of $G$-orbits in $X$.

For each orbit $C \in \orb X$, let $i_C: \overline C \hto X$ denote the inclusion morphism of the closure of $C$ as a reduced closed subscheme, and let $\cI_C \subset \cO_X$ denote the corresponding ideal sheaf.

\begin{rmk}\label{rmk:genl}
Some earlier references on staggered sheaves, including most of~\cite{a} and a significant part of~\cite{at:bs}, imposed much weaker hypotheses: the setting was a scheme of finite type over some noetherian base scheme admitting a dualizing complex, acted on by an affine group scheme over the same base, with no assumption on the number of orbits.  In the present paper, only the results of Sections~\ref{sect:pure} and~\ref{sect:purep} hold in such great generality. The main results do not, and it simplifies the discussion to impose these conditions at the outset.
\end{rmk}

We uniformly adopt the convention that terms like ``open subscheme,'' ``closed subscheme,'' and ``irreducible'' are always to be interpreted in a $G$-invariant sense.  That is, ``open subscheme'' should always be understood to mean ``$G$-invariant open subscheme,'' and a subscheme is ``irreducible'' if it is not a union of two proper closed ($G$-invariant) subschemes.

Let $\cg X$ denote the category of $G$-equivariant coherent sheaves on $X$,
and let $\Db X$ (resp.~$\Dm X$, $\Dp X$) denote the full subcategory of the
bounded (resp.~bounded above, bounded below) derived category of
$G$-equivariant quasicoherent sheaves on $X$ consisting of objects with
coherent cohomology.  It is well-known that $\Db X$ and $\Dm X$ are
equivalent to bounded and bounded-above derived categories of $\cg X$,
respectively.  As usual, we let $\Dl Xn$ and $\Dml Xn$ denote the
subcategories of $\Db X$ and $\Dm X$, respectively, consisting of objects
$\cF$ with $h^k(\cF) = 0$ for $k > n$.  $\Dg Xn$ and $\Dpg Xn$ are defined
similarly.  We also have the truncation functors
\[
\Tl n: \Dp X \to \Dl Xn
\qquad\text{and}\qquad
\Tg n: \Dm X \to \Dg Xn.
\]

Over the course of this paper, we will consider a rather large number of different kinds of subcategories of $\Db X$, all of which are denoted by decorating the symbol ``$\Db X$'' with various left and right super- and subscripts.  To avoid confusion, it is helpful to visualize these subcategories as various regions in a large $3$-dimensional grid in which the vertical axis represents cohomological degree in $\Db X$.  (See Section~\ref{sect:baric} and Section~\ref{sect:purep} for the meanings of the other axes.)  Thus, the standard $t$-structure and its heart may be pictured as follows:
\[
\Dl X0:\incgr{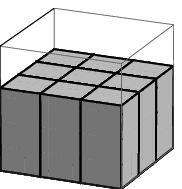}
\qquad\qquad
\Dg X0:\incgr{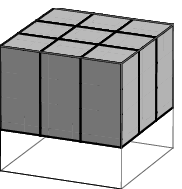}
\qquad\qquad
\cg X :\incgr{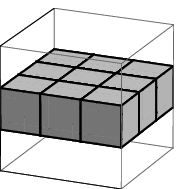}
\]

\subsection{Duality and codimension}
\label{sect:duality}

By~\cite[Proposition~1]{bez:pc}, $X$ possesses an equivariant Serre--Grothendieck dualizing complex.  Choose one, once and for all, and denote it by $\omega_X$.  We denote the Serre-Grothendieck duality functor by $\D = \cRHom(\cdot, \omega_X)$.

For each orbit $C \in \orb X$, there is a unique integer
\[
\cod C
\qquad\text{such that}\qquad
Ri_C^!\omega_X|_C \in \Dl C{\cod C} \cap \Dg C{\cod C}.
\]
This integer differs from the ordinary Krull codimension of $C$ by some constant depending only on $\omega_X$.  (See~\cite[Section~V.3]{har} and~\cite[Section~6]{a}.)  Thus, $\cod Y$ can be made to agree with the ordinary codimension by replacing $\omega_X$ by a suitable shift, but we do not assume here that any such specific normalization has been made.

\subsection{$s$-structures and altitude}
\label{sect:sstruc}

Suppose $\cg X$ is equipped with an increasing filtration $\{\cgl Xw\}_{w \in \Z}$ by Serre subcategories.  Let
\[
\cgg Xw = \{ \cG \in \cg X \mid
\text{$\Hom(\cF,\cG) = 0$ for all $\cF \in \cgl X{w-1}$} \}.
\]
For any sheaf $\cF \in \cg X$ and any integer $w \in \Z$, there is a unique maximal subsheaf of $\cF$ in $\cgl Xw$, denoted $\Sl w\cF$.  Conversely, the sheaf $\Sg {w+1}\cF = \cF/\Sl w\cF$ is the unique largest quotient of $\cF$ lying in $\cgg X{w+1}$.

The categories $(\{\cgl Xw\}, \{\cgg Xw\})_{w \in \Z}$ constitute an \emph{$s$-structure} on $X$ if they satisfy a rather lengthy list of axioms given in~\cite{a}, mostly having to do with $\Ext$-vanishing conditions on closed subschemes.  We will not review the full definition in the general case here, but we will give an explicit description of a certain class of $s$-structures below.

If $X$ is endowed with an $s$-structure, a sheaf $\cF \in \cg X$ is said to be \emph{$s$-pure of step $w$} if it lies in $\cgl Xw \cap \cgg Xw$.  (In~\cite{a}, this property was simply called ``pure,'' but here we call it ``$s$-pure'' to avoid confusion with the notions of baric and skew purity, {\it cf.} Section~\ref{sect:baric}.)  An $s$-structure on $X$ induces $s$-structures on all locally closed subschemes of $X$, and in particular on all orbits.

Given an orbit $C \in \orb X$, recall that $Ri_C^!\omega_X[\cod C]|_C$ lies in $\cg C$ (that is, it is concentrated in cohomological degree $0$).  According to~\cite[Section~6]{a}, there is a unique integer
\[
\alt C
\qquad\text{such that}\qquad
Ri_C^!\omega_X[\cod C]|_C \in \cgl C{\alt C} \cap \cgg C{\alt C}.
\]
This integer is called the \emph{altitude} of $C$.  Finally, the \emph{staggered codimension} of $C$, denoted $\scod C$, is defined by
\[
\scod C = \alt C + \cod C.
\]

Let us now return to the question of how to construct an $s$-structure.  Consider the special case where $X$ is a reduced scheme consisting of a single $G$-orbit.  In this case, the conditions for the collection $(\{\cgl Xw\}, \{\cgg Xw\})_{w \in \Z}$ to constitute an $s$-structure reduce to the following much simpler conditions:
\begin{enumerate}
\item For every sheaf $\cF \in \cg X$, there exist integers $v, w$ such that $\cF \in \cgg Xv \cap \cgl Xw$.
\item If $\cF \in \cgl Xw$ and $\cG \in \cgl Xv$, then $\cF \otimes \cG \in \cgl X{w+v}$.
\item If $\cF \in \cgg Xw$ and $\cG \in \cgg Xv$, then $\cF \otimes \cG \in \cgg X{w+v}$.
\end{enumerate}
In Section~\ref{sect:orbit}, we will give a constructive classification of all $s$-structures on a single orbit.

Now, suppose $X$ contains more than one orbit, and assume that each orbit is endowed with an $s$-structure.  Assume also that the following condition holds:
\begin{equation}\label{eqn:rec}
\text{For each orbit $C \subset X$, the sheaf $i_C^*\cI_C|_C$ is in $\cgl C{-1}$.}
\end{equation}
(The sheaf in question is simply the conormal bundle of $C$.)  By~\cite[Theorem~1.1]{as:flag}, the condition~\eqref{eqn:rec} implies that there is a unique $s$-structure on $X$ whose restriction to each orbit coincides with the given $s$-structure on that orbit.  In practice, the easiest way to produce explicit examples of $s$-structures seems to be to specify one on each orbit and then invoke~\cite[Theorem~1.1]{as:flag}.

Not every $s$-structure on $X$ arises in this way, but every $s$-structure for which condition~\eqref{eqn:rec} holds does.  Following~\cite{at:bs}, $s$-structures with this property are said to be \emph{recessed}.  

We assume for the remainder of the paper that $X$ is endowed with a fixed recessed $s$-structure.  For examples, see~\cite{as:flag, t}.

\subsection{Perversities}
\label{sect:perv}

A \emph{perversity} (or \emph{perversity function}) is simply a function $q: \orb X \to \Z$.  A perversity $q$ is said to be \emph{monotone} if whenever $C' \subset \overline C$, we have $q(C') \ge q(C)$.

A number of constructions in the sequel depend on the choice of a perversity.  We will often refer to specific kinds of perversities, such as ``baric perversities,'' ``Deligne--Bezrukavnikov perversities,'' and ``staggered perversities.'' These are not intrinsically different kinds of objects; rather, the adjectives serve merely to indicate how a particular perversity will be used ({\it e.g.}, to construct a baric structure).

Given a perversity $q: \orb X \to \Z$, we define three different kinds of ``dual perversity,'' as follows:
\begin{align*}
&\text{\emph{baric} dual:} & \dualq(C) &= 2\alt C - q(C) \\
&\text{\emph{Deligne--Bezrukavnikov} dual:} & \dualpforq(C) &= \cod C -
q(C) \\
&\text{\emph{staggered} dual:} & \bar q(C) &= \scod C - q(C)
\end{align*}
A perversity is called \emph{comonotone} if its dual is monotone.  This condition is, of course, ambiguous, but the intended type of duality will be clear from context whenever this term is used.

The \emph{middle perversity} of a given kind (baric,
Deligne--Bezrukavnikov, or staggered) is the unique perversity that is
equal to its own dual.  Clearly, the middle baric perversity is given by
\[
q(C) = \alt C.
\]
Similarly, the middle Deligne--Bezrukavnikov and staggered perversities,
when they exist, are given by the formulas
\[
q(C) = \half \cod C
\qquad\text{and}\qquad
q(C) = \half\scod C,
\]
respectively.  However, these formulas make sense only when all $\cod C$
or all $\scod C$, respectively, are even.

\subsection{Baric structures}
\label{sect:baric}

Following~\cite{at:bs}, a \emph{baric structure} on a triangulated category $\fD$ is a pair of
collections of thick subcategories $(\{\fD_{\le w}\}, \{\fD_{\ge w}\})_{w \in \Z}$ satisfying the following axioms:
\begin{enumerate}
\item $\fD_{\le w} \subset \fD_{\le w+1}$ and $\fD_{\ge w} \supset \fD_{\ge w+1}
$ for all $w$.
\item $\Hom(A,B) = 0$ whenever $A \in \fD_{\le w}$ and $B \in \fD_{\ge w+1}$.
\item For any object $X \in D$, there is a distinguished triangle $A \to X
\to B \to$ with $A \in \fD_{\le w}$ and $B \in \fD_{\ge w+1}$.
\item For any object $X \in D$, there exist integers $v, w$ such that $X \in \fD_{\ge v} \cap \fD_{\le w}$.
\end{enumerate}
(The last axiom was not part of the definition of ``baric structure'' in~\cite{at:bs}; rather, a baric structure satisfying this extra condition was called \emph{bounded}.  In this paper, however, all baric structures will be bounded.)  Given a baric structure on $\fD$, the inclusion functor $\fD_{\le w} \hto \fD$ admits a right adjoint, denoted $\Bl w$, and the inclusion $\fD_{\ge w} \hto \fD$ admits a left adjoint $\Bg w$.  The functors $\Bl w$ and $\Bg w$ are called \emph{baric truncation functors}.  For any object $X$ and any $w \in \Z$, there is a distinguished triangle
\[
\Bl w X \to X \to \Bg {w+1}X \to,
\]
and any distinguished triangle as in Axiom~(3) above is canonically isomorphic to this one.

The main result of~\cite{at:bs} was the construction of a family of baric structures on $\Db X$, which we now recall.  Let $q: \orb X \to \Z$ be a perversity.  We define a full subcategory of $\cg X$ as follows:
\begin{equation}\label{eqn:qcgl}
\q\cgl Xw = \{ \cF \in \cg X \mid
\text{$i_C^*\cF|_C \in \cgl 
C{\lfloor \frac{w+q(C)}{2}\rfloor}$ for all $C \in \orb X$} \}.
\end{equation}
Note that this does not agree with the definition in~\cite{at:bs}: in {\it loc.~cit.}, pullbacks to orbits were required to lie in $\cgl C{w+q(C)}$, not $\cgl C{\lfloor(w+q(C))/2\rfloor}$.  Thus, the relationship between the two definitions is as follows:
\[
\text{$\q\cgl Xw$ as in~\cite{at:bs}} =
\text{$\tql\cgl X{2w}$ as in the present paper.}
\]
(The reason for this change will be explained below.)  Next, let
\begin{equation}\label{eqn:baric-defn}
\begin{aligned}
\q\Dmsl Xw &= \{ \cF \in \Dm X \mid 
\text{$h^k(\cF) \in \q\cgl Xw$ for all $k$} \}, \\
\q\Dpsg Xw &= \{ \cF \in \Dp X \mid
\text{$\Hom(\cG,\cF) = 0$ for all $\cG \in \q\Dmsl X{w-1}$} \}.
\end{aligned}
\end{equation}
Let $\q\Dsl Xw$ and $\q\Dsg Xw$ denote the bounded versions of these
categories, {\it i.e.}, the intersections of the categories above with $\Db
X$.  According to~\cite[\thmbaric]{at:bs}, $(\{\q\Dsl Xw\}, \{\q\Dsg Xw\})_{w\in \Z}$
is a baric structure on $\Db X$.  We write $\q\Bl w$ and $\q\Bg w$ for its
baric truncation functors, and we let
\[
\q\Dsp Xw = \q\Dsl Xw \cap \q\Dsg Xw.
\]
$\q\Dsp Xw$ is a full triangulated subcategory of $\Db X$.  Its objects are said to be \emph{pure of baric degree $w$} (with respect to the baric perversity $q$).  Note that for a sheaf in $\cg X$, there is no concise relationship between purity and $s$-purity: neither condition implies the other.

In the $3$-dimensional grid picture of $\Db X$, the horizontal axis represents baric degree.  Thus, the various categories associated to a baric structure may be drawn as follows:
\[
\q\Dsl Xw:\incgr{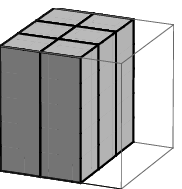}
\qquad\qquad
\q\Dsg Xw:\incgr{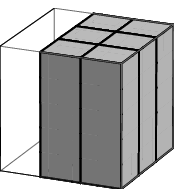}
\qquad\qquad
\q\Dsp Xw:\incgr{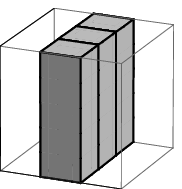}
\]
Observe that the category $\q\cgl Xw$ is simply $\cg X \cap \q\Dsl Xw$. 
We draw it thus:
\[
\q\cgl Xw:\incgr{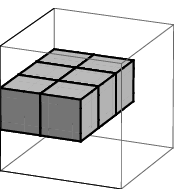}.
\]
However, it would be misleading to draw a similar picture of $\cg X \cap
\q\Dsg Xw$, because $\cg X$ is not, in general, generated by the
subcategories $\q\cgl Xw$ and $\cg X \cap \q\Dsg Xw$.  The latter category
does not seem to have very good properties, and it will not make an
appearance in the sequel.  (See~\cite{at:bs} for more information about
this category.)

The following useful result states that these baric
structures are both \emph{hereditary} (well-behaved on closed subschemes)
and \emph{local} (well-behaved on open subschemes).

\begin{lem}[{\cite[\lembaricres]{at:bs}}]\label{lem:baric-res}
Let $j: U \hto X$ be the inclusion of an open subscheme, and $i: Z \hto X$
the inclusion of a closed subscheme.  Then:
\begin{enumerate}
\item $j^*$ takes $\q\Dmsl Xw$ to $\q\Dmsl Uw$ and $\q\Dpsg Xw$ to
$\q\Dpsg Uw$.
\item $Li^*$ takes $\q\Dmsl Xw$ to $\q\Dmsl Zw$.
\item $Ri^!$ takes $\q\Dpsg Xw$ to $\q\Dpsg Zw$.
\item $i_*$ takes $\q\Dmsl Zw$ to $\q\Dmsl Xw$ and $\q\Dpsg Zw$ to
$\q\Dpsg Xw$.\qed
\end{enumerate}
\end{lem}

By applying the duality functor $\D$ to the categories that constitute
some given baric structure on $\Db X$, one can obtain a new baric
structure, said to be \emph{dual} to the given
one.  It follows from the construction in~\cite[\sectbariccoh]{at:bs} that
the dual baric structure to
$(\{\q\Dsl Xw\}, \{\q\Dsg Xw\})_{w\in \Z}$ is the baric structure associated by the
above formulas to the dual baric perversity:
\[
\D(\q\Dsl Xw) = \cheq\Dsg X{-w}
\qquad\text{and}\qquad
\D(\q\Dsg Xw) = \cheq\Dsl X{-w}.
\]
In particular, if $q$ is the middle baric perversity $q(C) = \alt C$,
then the baric structure $(\{\q\Dsl Xw\}, \{\q\Dsg Xw\})_{w\in \Z}$ is self-dual.  We
adopt the convention that when the left-subscript perversity is omitted,
this self-dual baric structure is meant:
\begin{align*}
\Dsl Xw &= \text{$\q\Dsl Xw$ with respect to $q(C) = \alt C$,} \\
\Dsg Xw &= \text{$\q\Dsg Xw$ with respect to $q(C) = \alt C$.}
\end{align*}
From Section~\ref{sect:purity} on, we will work almost exclusively with this self-dual baric structure.

\begin{rmk}
The existence of a self-dual baric structure is why the definition of $\q\cgl Xw$ was changed from that in~\cite{at:bs}: in the notation of {\it loc.~cit.}, the definitions~\eqref{eqn:baric-defn} can give rise to a self-dual baric structure only if $\alt C$ is even for all $C \in \orb X$.  Here, we do not wish to impose that restriction on the $s$-structure, and we circumvent it by modifying the definition of $\q\cgl Xw$.
\end{rmk}

\subsection{Staggered $t$-structures}
\label{sect:stagt}

Let $q: \orb X \to \Z$ be a perversity.  We define full subcategories of $\Dm X$ and $\Dp X$ as follows:
\begin{align*}
\qu\Dml Xn &=
\{ \cF \in \Dm X \mid
\text{$h^k(\cF) \in \tql\cgl X{n-2k}$ for all $k$} \}, \\
\qu\Dpg Xn &=
\{ \cF \in \Dp X \mid
\text{$\Hom(\cG,\cF) = 0$ for all $\cG \in \qu\Dmlt Xn$} \}.
\end{align*}
We also write $\qu\Dl Xn$ and $\qu\Dg Xn$ for the corresponding bounded categories:
\[
\qu\Dl Xn = \qu\Dml Xn \cap \Db X,
\qquad
\qu\Dg Xn = \qu\Dpg Xn \cap \Db X.
\]
We may draw pictures of these categories as follows:
\[
\qu\Dl X0:\incgr{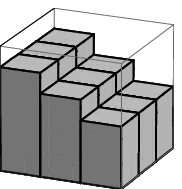}
\qquad\qquad
\qu\Dg X0:\incgr{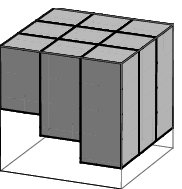}
\]
Although these pictures are useful, care should be exercised in
interpreting them.  In particular, for a fixed perversity $q$, it does not
make sense to superimpose the picture for, say, $\q\Dsl Xw$ with that for
$\qu\Dl X0$, because in the latter, the horizontal axis represents baric
degree with respect to the baric perversity $2q$, not $q$.  Also, the
picture above for $\qu\Dg X0$ may be interpreted as saying that for each
$k$, we have $\Dg Xk \cap \tql\Dsg X{-2k} \subset \qu\Dg X0$.  It does
\emph{not} say that $h^k(\cF) \in \tql\Dsg X{-2k}$ for $\cF \in \qu\Dg X0$;
indeed, the latter condition is false in general.

According to~\cite[\thmstagcoh]{at:bs}, $(\qu\Dl X0, \qu\Dg X0)$ is a
bounded, nondegenerate $t$-structure on $\Db X$.  Moreover, its heart
\[
\qu\cM(X) = \qu\Dl X0 \cap \qu \Dg X0,
\]
known as the category of \emph{staggered sheaves} (of perversity $q$), is a finite-length category.  Its truncation functors are denoted
\[
\qu\Tl n : \Db X \to \qu\Dl Xn
\qquad\text{and}\qquad
\qu\Tg n : \Db X \to \qu\Dg Xn,
\]
and the associated cohomology functors are denoted $\qu h^n: \Db X \to \qu\cM(X)$. 

The simple objects in this category are parametrized by pairs $(C,\cL)$, where $C \in \orb X$, and $\cL$ is an irreducible vector bundle on $C$.  To describe the structure of the corresponding simple object, we require the notion of the \emph{intermediate-extension functor}.  This is a fully faithful functor
\[
\qu j^C_{!*}: \qu\cM(C) \to \qu\cM(\overline C)
\]
that takes an object $\cF \in \qu\cM(C)$ to the unique object of $\qu\cM(\overline C)$ with the following properties:
\begin{enumerate}
\item $\qu j^C_{!*}\cF|_C \cong \cF$;
\item For any smaller orbit $C' \subset \overline C \ssm C$, we have
$Li_{C'}^* \qu j^C_{!*}\cF \in \qu\Dml {\overline{C'}}{-1}$ and $Ri_{C'}^!
\qu j^C_{!*}\cF \in \qu \Dpg {\overline{C'}}{1}$.
\end{enumerate}
Now, an irreducible vector bundle $\cL \in \cg C$ is necessarily $s$-pure;
suppose it is $s$-pure of step $v$.  Then $\cL[v - q(C)]$ is an object of
$\qu\cM(C)$, and the object
\[
\qu\cIC(\overline C, \cL[v - q(C)]) = i_{C*} j^C_{!*}\cL[v - q(C)],
\]
known as a (\emph{staggered}) \emph{intersection cohomology complex}, is a
simple object of $\qu\cM(X)$.  Every simple object of $\qu\cM(X)$ arises in
this way.

As with baric structures, there is an easy description of the dual
$t$-structure to a given staggered $t$-structure: according
to~\cite[\thmstagdual]{at:bs}, it is the staggered $t$-structure associated
to
the dual staggered perversity.  That is,
\[
\D(\qu\Dl Xn) = \barq\Dg X{-n}
\qquad\text{and}\qquad
\D(\qu\Dg Xn) = \barq\Dl X{-n}.
\]
In particular, if $\scod C$ is even for all $C \in \orb X$, then the
staggered $t$-structure associated to the middle staggered perversity
$q(C) = \half\scod C$ is self-dual.

\subsection{Sheaves on nonreduced schemes}
\label{sect:nonred}

We conclude with a useful lemma comparing various categories of sheaves on
a nonreduced scheme with those on its associated reduced scheme.

\begin{lem}\label{lem:red}
Let $X_\red$ denote the reduced scheme associated to $X$, and let $t:
X_\red \hto X$ be the natural map.  Let $q: \orb X \to \Z$ be a perversity.
\begin{enumerate}
\item If $\cF \in \cg X$, we have $\cF \in \cgl Xw$ if and only if $t^*\cF
\in \cgl {X_\red}w$.
\item If $\cF \in \cg X$, we have $\cF \in \q\cgl Xw$ if and only if
$t^*\cF \in \q\cgl {X_\red}w$.
\item If $\cF \in \Dm X$, we have $\cF \in \Dml Xn$ if and only if $Lt^*\cF
\in \Dml {X_\red}n$.
\item If $\cF \in \Dm X$, we have $\cF \in \q\Dmsl Xw$ if and only if
$Lt^*\cF \in \q\Dmsl {X_\red}w$.
\end{enumerate}
\end{lem}
There is a dual version of this lemma involving ``$\ge$'' categories and
the $t^!$ and $Rt^!$ functors, but this statement suffices for our needs.
\begin{proof}
Part~(1) is contained in~\cite[Proposition~4.1]{a}, and part~(4) is
contained in~\cite[\proprigid]{at:bs}.  Part~(2) is obvious from the
definition.

It remains to prove part~(3).  If $\cF \in \Dml Xn$, then clearly $Lt^*\cF
\in \Dml {X_\red}n$, since $Lt^*$ is right $t$-exact.  Conversely, suppose
$\cF \notin \Dml Xn$. Let $k$ be the largest integer such that $h^k(\cF)
\ne 0$. Of course, we have $k > n$.  By applying $Lt^*$ to the
distinguished triangle
\[
\Tlt k\cF \to \cF \to h^k(\cF)[-k] \to
\]
and then forming the cohomology long exact sequence, one sees that
$h^k(Lt^*\cF) \cong t^*h^k(\cF)$.  The functor $t^*$ kills no nonzero
sheaf, so $h^k(Lt^*\cF) \ne 0$, and hence $Lt^*\cF \notin \Dml {X_\red}n$.
\end{proof}

\section{Pure Sheaves}
\label{sect:pure}

Let $q: \orb X \to \Z$ be a baric perversity.  The category of $\q\Dsp Xw$
of pure objects is not stable under the standard truncation functors, so
the standard $t$-structure on $\Db X$ does not induce a $t$-structure on
$\q\Dsp Xw$.  Our goal in this section is to find an ``easy'' $t$-structure
on $\q\Dsp Xw$ that resembles the standard $t$-structure on $\Db X$ as
closely as possible.

Let us define full subcategories of $\Dm X$ and $\Dp X$ by
\begin{align*}
\q\Dmll Xnw &= (\Dml Xn * \q\Dmslt Xw) \cap \q\Dmsl Xw, \\
\q\Dpgg Xnw &= \Dpg Xn \cap \q\Dpsg Xw. 
\end{align*}
(For the ``$*$'' operation on triangulated categories, see~\cite[\S
1.3.9]{bbd}.)  Note that the definition of $\q\Dmll Xnw$ involves the
condition ``$<\!w$,''
{\it sic}.  Let $\q\Dll Xnw$ and $\q\Dgg Xnw$ denote the bounded versions
of these categories, {\it i.e.}, the intersections of the above categories
with $\Db X$.  These categories may be pictured as follows:
\[
\q\Dll Xnw: \incgr{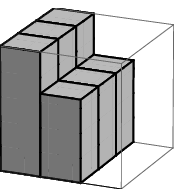}
\qquad\qquad
\q\Dgg Xnw: \incgr{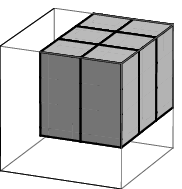}
\]
Finally, we denote the intersections of these categories with the category
$\q\Dsp Xw$ of pure objects by
\[
\q\Dlp Xnw = \q\Dmll Xnw \cap \q\Dsp Xw
\quad\text{and}\quad
\q\Dgp Xnw = \q\Dpgg Xnw \cap \q\Dsp Xw,
\]
and we draw them thus:
\[
\q\Dlp Xnw: \incgr{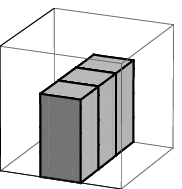}
\qquad\qquad
\q\Dgp Xnw: \incgr{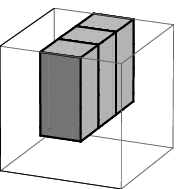}
\]
The pictures suggest that $(\q\Dlp X0w, \q\Dgp X0w)$ is a $t$-structure on
$\q\Dsp Xw$.  The main result of this section states that this is, in fact,
the case.

\begin{lem}\label{lem:pb-pure-exact}
Let $j: U \hto X$ be the inclusion of an open subscheme, and $i: Z \hto X$
the inclusion of a closed subscheme.  Then:
\begin{enumerate}
\item $j^*$ takes $\q\Dmll Xnw$ to $\q\Dmll Unw$ and $\q\Dpgg Xnw$ to
$\q\Dpgg Unw$.
\item $Li^*$ takes $\q\Dmll Xnw$ to $\q\Dmll Znw$.
\item $Ri^!$ takes $\q\Dpgg Xnw$ to $\q\Dpgg Znw$.
\item $i_*$ takes $\q\Dmll Znw$ to $\q\Dmll Xnw$ and $\q\Dpgg Znw$ to
$\q\Dpgg Xnw$.
\end{enumerate}
\end{lem}
\begin{proof}
Immediate from Lemma~\ref{lem:baric-res} and well-known $t$-exactness
properties of these functors with respect to the standard $t$-structure.
\end{proof}

\begin{lem}\label{lem:pure-hom}
If $\cF \in \q\Dmll Xnw$ and $\cG \in \q\Dpgtg Xnw$, then
$\Hom(\cF,\cG) = 0$.  Conversely, if $\cF \in \q\Dmsl Xw$ and
$\Hom(\cF,\cG) = 0$ for all $\cG \in \q\Dgtp Xnw$, then $\cF \in \q\Dmll
Xnw$.
\end{lem}
\begin{proof}
First, suppose $\cF \in \q\Dmll Xnw$, and find a distinguished triangle
$\cF' \to \cF \to \cF'' \to$ with $\cF' \in \Dml Xn$ and $\cF'' \in \Dmslt
Xw$.  If $\cG \in \q\Dpgtg Xnw$, then $\Hom(\cF',\cG) = \Hom(\cF'',\cG) =
0$, so we see that $\Hom(\cF,\cG) = 0$ as well.

On the other hand, given $\cF \in \q\Dmsl Xw$ such that $\Hom(\cF,\cG) =
0$ for all $\cG \in \q\Dgtp Xnw$, form the distinguished triangle
\[
\Tl n\cF \to \cF \to \Tgt n\cF \to.
\]
To show that $\cF \in \q\Dmll Xnw$, it suffices to show that $\Tgt n\cF
\in \q\Dslt Xw$.  Suppose this is not the case, and let $\cG = \q\Bg
w\Tgt n\cF$.  Since $\q\Bg w$ is left $t$-exact, we see that $\cG \in
\q\Dpgtg Xnw$.  Clearly, $\Hom((\Tl n\cF)[1], \cG) = 0$, so the fact that 
$\Hom(\Tgt n\cF, \cG) \ne 0$ implies that $\Hom(\cF,\cG) \ne 0$, a
contradiction. 
\end{proof}

\begin{lem}\label{lem:pure-dt}
For any $\cF \in \q\Dsl Xw$, there is a distinguished triangle $\cF' \to
\cF
\to \cF'' \to$ with $\cF' \in \q\Dll Xnw$ and $\cF'' \in \Dgtp Xnw$.
\end{lem}
\begin{proof}
Let $\cF'' = \q\Bg w\Tgt n\cF$.  Since $\Tgt n$ is right baryexact, we have
$\Tgt n\cF \in \Dgt Xn \cap \q\Dsl Xw$.  Then, the left $t$-exactness of
$\q\Bg w$ (with respect to the standard $t$-structure) implies that
\[
\cF'' = \q\Bg w\Tgt n\cF \in \Dgt Xn \cap \q\Dsl Xw \cap \q\Dsg Xw = \q\Dgp
Xnw.
\]
We also have a natural morphism $\cF \to \cF''$,
obtained by composing  $\cF \to \Tgt n\cF$ and $\Tgt
n\cF \to \q\Bg w\Tgt n\cF$.  Let $\cF'$ be the cocone of this morphism.  We
already know that $\cF' \in \q\Dsl Xw$.  The octahedral diagram below shows
that $\cF' \in \Dl Xn * \q\Dslt Xw$, so $\cF' \in \q\Dll Xnw$, as desired.
\[
\vcenter{\xymatrix@=10pt{
&&&& \Tgt n\cF \ar[ddddl]\ar[ddr]_{+1} \\ \\
&&&&& \Tl n\cF \ar[dlllll]|!{[uul];[ddll]}\hole
  \ar[ddddl]|!{[ddll];[drrr]}\hole \\
\cF \ar[uuurrrr]\ar[drrr] &&&&&&&& \q\Blt w\Tgt n\cF \ar[uuullll]
\ar[ulll]^{+1} \\
&&& \cF'' \ar[urrrrr]_{+1}\ar[ddr]^{+1} \\ \\
&&&& \cF' \ar[uuullll]\ar[uuurrrr]
}}\qedhere
\]
\end{proof}

\begin{prop}\label{prop:pure-t}
$(\q\Dlp X0w, \q\Dgp X0w)$ is a nondegenerate, bounded $t$-structure on
$\q\Dsp Xw$.
\end{prop}
\begin{proof}
It is clear that $\q\Dll X0w \subset \q\Dll X1w$ and $\q\Dgg X0w \supset
\q\Dgg X1w$.  Next, given $\cF \in \q\Dsp Xw$, form a distinguished
triangle $\cF' \to \cF \to \cF'' \to$ as in Lemma~\ref{lem:pure-dt}. 
According to that lemma, $\cF''$ necessarily lies in $\q\Dsp Xw$, so it
follows that $\cF'$ does as well.  From Lemma~\ref{lem:pure-hom}, we see
that $(\q\Dlp X0w, \q\Dgp X0w)$ is indeed a $t$-structure on $\q\Dsp Xw$.

It is clear that no nonzero object can belong to $\q\Dgp Xnw$, or even
$\q\Dgg Xnw$, for all $n$.  On the other hand, the only objects that belong
to $\q\Dll Xnw$ for all $n$ are those in $\q\Dslt Xw$, and only the zero
object lies in $\q\Dslt Xw \cap \q\Dsp Xw$.  Thus, this $t$-structure is
nondegenerate.  Its boundedness then follows from the boundedness of the
standard $t$-structure on $\Db X$.
\end{proof}

\begin{defn}
The $t$-structure of Proposition~\ref{prop:pure-t} is called the
\emph{purified standard $t$-structure}, or simply the \emph{purified
$t$-structure}, on $\q\Dsp Xw$.  Its truncation functors are denoted
\[
\q\Tlp nw: \q\Dsp Xw \to \q\Dlp Xnw
\qquad\text{and}\qquad
\q\Tgp nw: \q\Dsp Xw \to \q\Dgp Xnw.
\]
\end{defn}

\section{Pure-Perverse Coherent Sheaves}
\label{sect:purep}

Let $q: \orb X \to \Z$ be a function.  In this section, we construct a new
$t$-structure on the category $\q\Dsp Xw$ of pure objects, called the
\emph{pure-perverse $t$-structure}.  It is related to the purified standard
$t$-structure in the same way the perverse coherent $t$-structure
of~\cite{bez:pc} is related to the standard $t$-structure on $\Db X$.  We
then prove that the heart of the pure-perverse $t$-structure is contained
in the heart of a suitable staggered $t$-structure $(\ru\Dl X0, \ru\Dg
X0)$.  This is an important step towards the Purity Theorem, as it will
enable us to prove in the next section that a certain operation in the
heart of the staggered $t$-structure can be replaced by one in the heart of
the pure-perverse $t$-structure.

The construction of the pure-perverse $t$-structure closely follows the
construction of the perverse coherent $t$-structure in~\cite{bez:pc}.  As in {\it loc. cit.}, the pure-perverse $t$-structure depends on the choice
of a monotone and comonotone Deligne-Bezrukav\-nikov perversity, {\it i.e.}, a
function $p: \orb X \to \Z$ satisfying
\[
0 \le p(C') - p(C) \le \cod C' - \cod C
\]
whenever $C' \subset \overline C$.

Fix a monotone and comonotone Deligne--Bezrukavnikov perversity $p: \orb X
\to \Z$.  Define full subcategories of $\q\Dmsl Xw$ and $\q\Dpsg Xw$ as
follows:
\begin{gather*}
\pq\Dmll Xnw = \{ \cF \mid
\text{$Li_C^*\cF|_C \in \q\Dmll C{n+p(C)}w$ for all $C \in \orb X$} \} \\
\pq\Dpgg Xnw = \D(\bpcq\Dmll X{-n}{-w})
\end{gather*}
It follows from the gluing theorem for baric structures~\cite[\thmbaricglue]{at:bs} and induction on the number of orbits that $\pq\Dmll Xnw \subset \q\Dmsl Xw$, and hence that $\pq\Dpgg Xnw \subset \q\Dpsg Xw$.

The set $\orb X$ is, of course, partially ordered by inclusion.  Suppose
for a moment that this partial order is, in fact, a total order.  In this
case, we can draw pictures of the above categories similar to our pictures
of other subcategories of $\Db X$, by regarding the third axis of the grid
as representing orbits in $\orb X$, with larger orbits closer to the
reader, and smaller orbits father away.  Since $p$ takes larger values on
smaller orbits, we may draw the bounded versions of $\pq\Dmll Xnw$ and
$\pq\Dpgg Xnw$ thus:
\[
\pq\Dll Xnw: \incgr{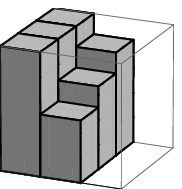}
\qquad\qquad
\pq\Dgg Xnw: \incgr{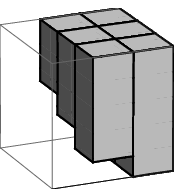}
\]
(The picture of $\pq\Dgg Xnw$ has been drawn from an unusual perspective to
make its structure visible.)  We will also work with the intersections of
these categories with the pure category $\q\Dsp Xw$:
\begin{align*}
\pq\Dlp Xnw = \pq\Dll Xnw \cap \q\Dsp Xw: \incgr{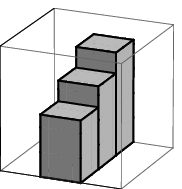}
\\
\pq\Dgp Xnw = \pq\Dgg Xnw \cap \q\Dsp Xw: \incgr{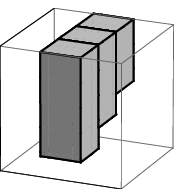}
\end{align*}
These pictures do not make much sense if $\orb X$ is not totally ordered,
but they may nevertheless be a helpful source of intuition.

\begin{lem}\label{lem:pb-purep-exact}
Let $j: U \hto X$ be the inclusion of an open subscheme, and $i: Z \hto X$
the inclusion of a closed subscheme.  Then:
\begin{enumerate}
\item $j^*$ takes $\pq\Dmll Xnw$ to $\pq\Dmll Unw$ and $\pq\Dpgg Xnw$ to
$\pq\Dpgg Unw$.
\item $Li^*$ takes $\pq\Dmll Xnw$ to $\pq\Dmll Znw$.
\item $Ri^!$ takes $\pq\Dpgg Xnw$ to $\pq\Dpgg Znw$.
\item $i_*$ takes $\pq\Dmll Znw$ to $\pq\Dmll Xnw$ and $\pq\Dpgg Znw$ to
$\pq\Dpgg Xnw$.
\end{enumerate}
\end{lem}
\begin{proof}
Parts~(1) and~(2) are immediate from the definition of $\pq\Dmll Xnw$, and
part~(3) follows by duality.  Similarly, because $i_*$ commutes with $\D$,
the second part of part~(4) follows from the first part.

It remains to show that if $\cF \in \pq\Dmll Znw$, then $i_*\cF \in
\pq\Dmll Xnw$.  We must show that for any orbit $C \in \orb X$,
$Li_C^*i_*\cF|_C \in \q\Dmll C{n+p(C)}w$.  In fact, it suffices to
consider the case where $C$ is a closed orbit contained in $Z$: if $C
\not\subset Z$, then $Li_C^*i_*\cF|_C = 0$, and if $C$ is not closed,
the operation $Li_C^*(\cdot)|_C$ factors as restriction to the open subscheme $V = X \ssm
(\overline C \ssm C)$ followed by pullback to the closed subscheme $C
\subset V$, and we already know by part~(1) that restriction to $V$ takes
$\pq\Dmll Xnw$ to $\pq\Dmll Vnw$.

Assume, therefore, that $C$ is a closed orbit contained in $Z$.  If $\cF \in\pq\Dmll Znw$ but
$Li_C^*i_*\cF \notin \q\Dmll C{n+p(C)}w$, then, by
Lemma~\ref{lem:pure-hom}, there exists an object $\cG \in \q\Dgtp
C{n+p(C)}w$ such that $\Hom(Li_C^*i_*\cF, \cG) \ne 0$.  By adjunction, it follows that $\Hom(\cF, Ri^! i_{C*}\cG) \ne 0$, and by
Lemma~\ref{lem:pb-pure-exact}, we have $Ri^! i_{C*}\cG \in \q\Dpgtg
Z{n+p(C)}w$.  Now, let $W = Z \ssm C$, and consider the exact sequence
\[
\lim_{\substack{\to \\ Z'}} \Hom(Li^*_{Z'}\cF, Ri^!_{Z'}Ri^!i_{C*}\cG) \to
\Hom(\cF,Ri^!i_{C*}\cG) \to \Hom(\cF|_W,Ri^!i_{C*}\cG|_W),
\]
where $i_{Z'}: Z' \hto Z$ ranges over all closed subscheme structures on
$C \subset Z$.  (For an explanation of this exact sequence, see, for
instance, the proof of~\cite[Proposition~2]{bez:pc}.  Similar sequences
will be used in Lemmas~\ref{lem:purep-hom} and~\ref{lem:purep-stag} and 
in Proposition~\ref{prop:skew-orth}.)  The last term vanishes since
$Ri^!i_{C*}\cG|_W = 0$. Moreover we have $Li^*_{Z'}\cF \in \q\Dmll {Z'}{n+p(C)}w$ for any
subscheme structure, by Lemma~\ref{lem:red}. On the other hand,
$Ri_{Z'}^!Ri^!i_{C*}\cG \in \q\Dpgtg {Z'}{n+p(C)}w$ by
Lemma~\ref{lem:pb-pure-exact}, so the first term above vanishes by
Lemma~\ref{lem:pure-hom}.  Thus, the middle term vanishes as well, a
contradiction.  Therefore, $i_*\cF \in \pq\Dmll Xnw$.
\end{proof}

\begin{lem}\label{lem:purep-hom-pre}
Let $d$ be the minimum value of $\cod C$ over all $C \in \orb X$.  If $\cF \in \q\Dmll Xnw$ and $\cG \in \cheq\Dmltl
X{d-n}{-w}$, then $\Hom(\cF,\D\cG) = 0$.
\end{lem}
\begin{proof}
We know, by the definition of $\cheq\Dmltl X{d-n}{-w}$, that there is a
distinguished triangle $\cG' \to \cG \to \cG'' \to$ with $\cG' \in \Dmlt
X{d-n}$ and $\cG'' \in \cheq\Dmslt X{-w}$.  The fact that $\cG \in
\cheq\Dmsl X{-w}$ implies that $\cG' \in \cheq\Dmsl X{-w}$ as well. 
Applying $\D$, we obtain a distinguished triangle
\[
\D\cG'' \to \D\cG \to \D\cG' \to.
\]
Note that $\D\cG'' \in \q\Dpsgt Xw$.  Since $\cF \in \q\Dmsl Xw$, we see
that $\Hom(\cF,\D\cG'') = \Hom(\cF,\D\cG''[1]) = 0$, so $\Hom(\cF,\D\cG)
\cong \Hom(\cF,\D\cG')$.  Now, $\cF$ arises in some distinguished triangle
\[
\cF' \to \cF \to \cF'' \to
\]
with $\cF' \in \Dml Xn$ and $\cF'' \in \q\Dmslt Xw$.  Note that the
definition of $d$ is such that $\D(\Dmlt X{d-n}) \subset \Dpgt Xn$. 
Therefore, we see that $\D\cG' \in \q\Dpsg Xw \cap \Dpgt Xn$.  It follows
that $\Hom(\cF',\D\cG') = 0$ and $\Hom(\cF'',\D\cG') = 0$.  We conclude
that $\Hom(\cF,\D\cG') = 0$, and hence that $\Hom(\cF,\D\cG) = 0$, as
desired.
\end{proof}

\begin{lem}\label{lem:purep-hom}
If $\cF \in \pq\Dmll Xnw$ and $\cG \in \pq\Dpgtg Xnw$, then $\Hom(\cF,\cG)
= 0$.
\end{lem}
\begin{proof}
We proceed by noetherian induction, and assume the statement is known on
all proper closed subschemes of $X$.  Let $\cG' = \D\cG \in \bpcq\Dmltl
X{-n}{-w}$.  Choose an open orbit $C \in \orb X$, and let $U \subset X$ be
the corresponding open subscheme.  By Lemma~\ref{lem:red}, $\cF|_U \in
\q\Dmll U{n+p(C)}w$ and $\cG'|_U \in \cheq\Dmltl U{-n+\dualp(C)}{-w}$.  Of
course, $-n + \dualp(C) = \cod C - (n+p(C))$, so by
Lemma~\ref{lem:purep-hom-pre}, $\Hom(\cF|_U,\D\cG'|_U) = 0$.  Now, let
$Z$ be the complementary closed subspace to $U$, and consider the exact
sequence 
\[
\lim_{\substack{\to \\ Z'}} \Hom(Li^*_{Z'}\cF, Ri^!_{Z'}\cG) \to
\Hom(\cF,\cG) \to \Hom(\cF|_U,\cG|_U),
\]
where $i_{Z'}: Z' \hto X$ ranges over all closed subscheme structures on
$Z$.  We have just seen that the last term vanishes.  Since $Li^*_{Z'}\cF
\in \pq\Dmll {Z'}nw$ and $Ri^!_{Z'}\cG \in \pq\Dpgtg {Z'}nw$, the first
term vanishes by induction.  So $\Hom(\cF,\cG) = 0$, as desired.
\end{proof}

\begin{prop}\label{prop:purep-t}
$(\pq\Dlp X0w, \pq\Dgp X0w)$ is a nondegenerate, bounded $t$-structure on
$\q\Dsp Xw$.
\end{prop}

\begin{defn}
The $t$-structure of Proposition~\ref{prop:purep-t} is called the \emph{pure-perverse $t$-structure}.  Its truncation functors are denoted
\[
\pq\Tlp nw: \q\Dsp Xw \to \pq\Dlp Xnw
\qquad\text{and}\qquad
\pq\Tgp nw: \q\Dsp Xw \to \pq\Dgp Xnw,
\]
and its heart, denoted $\pq\Pp Xw$, is called the category of \emph{pure-perverse coherent sheaves}.
\end{defn}

\begin{proof}
In view of Lemma~\ref{lem:purep-hom}, to show that these categories form a
$t$-structure, it remains only to show that for any $\cF \in \q\Dsp Xw$,
there is a distinguished triangle $\cF' \to \cF \to \cF'' \to$ with $\cF'
\in \pq\Dlp X0w$ and $\cF'' \in \pq\Dgtp X0w$.  Our argument closely
follows the proof of~\cite[Theorem~1]{bez:pc}.  Choose an open orbit $C
\in \orb X$ on which $p$ achieves its minimum value, and let $U \subset X$
be the corresponding open subscheme.  (The monotonicity of $p$ guarantees
that its minimum value is achieved on an open orbit.) Let $\cF_1 = \q\Tlp
{p(C)}w \cF$.  By Lemma~\ref{lem:pb-pure-exact} and the monotonicity of
$p$, we have that $\cF_1 \in \pq\Dlp X0w$.  Form the distinguished triangle
\[
\cF_1 \to \cF \to \cG_1 \to,
\]
where $\cG_1 = \q\Tgtp {p(C)}w\cF$.  
It is clear that $\D\cF \in \cheq\Dsl X{-w}$, and it follows
from~\cite[Lemmas~6.1 and~6.6]{a} that $\D\cF|_U \in \Dl X{\cod C-n}$, so that we have
$\D(\cG_1)|_U \in \cheq\Dltp X{\dualp(C)}{-w}$.   Therefore, in the
distinguished triangle
\[
\cheq\Tltp {\dualp(C)}{-w}(\D\cG_1) \to \D\cG_1 \to \cheq\Tgp {\dualp(C)}{-w}
\D\cG_1 \to,
\]
the support of the last term is contained in the complement of $U$.  Let
$\cG = \D(\cheq\Tgp {\dualp(C)}{-w} \D\cG_1)$ and $\cF_2 = \D(\cheq\Tltp
{\dualp(C)}{-w}(\D\cG_1))$.  Since $\dualp$ is monotone, $\cheq\Tltp
{\dualp(C)}{-w}(\D\cG_1) \in \bpcq\Dltp X0{-w}$, and therefore $\cF_2 \in
\pq\Dgtp X0w$.  We now have
\[
\cF \in \{\cF_1\} * \{\cG\} * \{\cF_2\},
\]
with $\cF_1 \in \pq\Dlp X0w$, $\cF_2 \in \pq\Dgtp X0w$, and $\cG$ supported
on a proper closed subscheme.  It follows by noetherian induction that
$(\pq\Dlp X0w, \pq\Dgp X0w)$ is a $t$-structure.  (See the proof
of~\cite[Theorem~1]{bez:pc} for more details on this argument.)

Next, let $d$ be the minimum value of $p$ on $X$, and let $e$ be its
maximum value.  We then have
\[
\q\Dlp Xdw \subset \pq\Dlp X0w \subset \q\Dlp Xew.
\]
Then the nondegeneracy and boundedness of the purified standard
$t$-structure imply that no nonzero object belongs to all $\pq\Dlp Xnw$,
and every object belongs to some $\pq\Dlp Xnw$.  By duality, corresponding
statements hold for $\pq\Dgp Xnw$ as well, so the $t$-structure $(\pq\Dlp
X0w, \pq\Dgp X0w)$ is nondegenerate and bounded.
\end{proof}

Under suitable conditions on the perversity function, it is possible to
define an ``intermediate-extension'' functor for pure-perverse coherent
sheaves, following the pattern of~\cite[Theorem~2]{bez:pc}.  Simple objects
in this category arise in this way, {\it cf.}~\cite[Corollary 4]{bez:pc}. 
In the next section (see Proposition~\ref{prop:purep-ic}), we will carry
out a slight generalization of this construction.

The remainder of the section is devoted to establishing a relationship
between pure-perverse coherent sheaves and staggered sheaves.

\begin{lem}\label{lem:purep-stag}
Suppose $\cF \in \pq\Dpgg X0w$.  Let $r: \orb X \to \Z$ be the function
$r(C) = p(C) + \lceil \frac{q(C) + w}{2} \rceil$.  Then $\cF \in \ru\Dpg
X0$.
\end{lem}
This statement can be thought of as saying that under a suitable change of
coordinates, we have
\[
\,\incgr{pdgg}\, \subset\, \incgr{qdg}\,.
\]
The ``change of coordinates'' is the change in the notion of baric degree
between the two pictures: the left-hand picture shows baric degree with
respect to $q$, and the right-hand picture shows baric degree with respect
to $2r \approx 2p + q + w$.

\begin{proof}
It suffices to show that $\Hom(\cG,\cF) = 0$ for all $\cG \in \ru\Dl
X{-1}$.  By induction on the number of nonzero cohomology sheaves of $\cG$,
we may assume without loss of generality that $\cG$ is concentrated in a
single degree: suppose $\cG \cong \cG'[n+1]$ for some sheaf $\cG' \in
\trl\cgl X{2n}$.

Choose an open orbit $C \in \orb X$, and let $U \subset X$ be the
corresponding open subscheme.  Then $\cG'|_U \in \cgl U{r(C)+n}$.  By
Lemma~\ref{lem:pb-purep-exact}, we have $\cF|_U \in \pq\Dpgg U0w =
\q\Dpgg U{p(C)}w$.  $\cG$ is concentrated in degree $-n-1$, so if $-n-1 <
p(C)$, we clearly have $\Hom(\cG|_U, \cF|_U) = 0$.  Now, assume $-n-1 \ge
p(C)$. It follows that
\[
r(C) + n = p(C) + \left\lceil \frac{q(C)+w}{2}\right\rceil + n 
\le \left\lceil \frac{q(C)+w}{2}\right\rceil - 1
= \left\lfloor \frac{q(C) + w - 1}{2}\right\rfloor.
\]
It follows that $\cG'|_U \in \cgl U{\lfloor(q(C)+w-1)/2\rfloor} = \q\cgl
U{w-1}$.  Thus, in this case, $\cG|_U \in \q\Dsl U{w-1}$, and we see once
again that $\Hom(\cG|_U, \cF|_U) = 0$.  The result then follows by
noetherian induction from the exact sequence
\[
\lim_{\substack{\to \\ Z'}} \Hom(Li^*_{Z'}\cG, Ri^!_{Z'}\cF) \to
\Hom(\cG,\cF) \to \Hom(\cF|_U,\cG|_U). \qedhere
\]
\end{proof}

\begin{prop}\label{prop:purep-stag}
Let $r: \orb X \to \Z$ be such that $p(C) + \lfloor \frac{q(C) + w}{2}
\rfloor \le r(C) \le p(C) + \lceil \frac{q(C) + w}{2} \rceil$.  Then
$\pq\Pp Xw \subset \ru\cM(X)$.
\end{prop}
\begin{proof}
Suppose $\cF \in \pq\Pp Xw$.  Let $r_1(C) = p(C) + \lceil \frac{q(C) +
w}{2} \rceil$.  The preceding lemma tells us that $\cF \in {}^{\sst r_1}\Dg
X0$. On the other hand, $\D\cF \in \bpcq\Pp X{-w}$, and invoking the
preceding lemma again tells us that $\D\cF \in {}^{\sst r_2}\Dg X0$, where
\begin{multline*}
r_2(C) = \dualp(C) + \left\lceil \frac{\dualq(C) - w}{2} \right\rceil
= \cod C - p(C) + \left\lceil \alt C - \frac{q(C) + w}{2}
\right\rceil \\
= \scod C - \left(p(C) + \left\lfloor \frac{q(C) + w}{2}
\right\rfloor\right).
\end{multline*}
By duality, we have $\cF \in {}^{\sst r_3}\Dl X0$, where $r_3(C) = p(C) +
\lfloor \frac{q(C) + w}{2} \rfloor$.  Thus, for any $r: \orb X \to \Z$ with
$r_3(C) \le r(C) \le r_1(C)$, we have $\cF \in \ru\cM(X)$.
\end{proof}

\section{Intermediate-Extension Functors}
\label{sect:ic}

In the previous section, we proved that every pure-perverse coherent sheaf
is a staggered sheaf with respect to a suitable staggered perversity.  In
this section, we will prove a kind of converse to this: we will show that
every simple staggered sheaf is pure-perverse with respect to suitable
Deligne--Bezrukavnikov and baric perversities.

Fix an orbit $C_0$, and let $j: C_0 \hto \overline C_0$ denote the
inclusion.  We define a staggered perversity $\ufl r: \orb X \to \Z$ by
\[
\ufl r(C) =
\begin{cases}
r(C) - 1 & \text{if $\overline C \subsetneq \overline C_0$,} \\
r(C) & \text{otherwise.}
\end{cases}
\]

Next, we define an open subscheme $\tilde C_0 \subset \overline C_0$ by
\[
\tilde C_0 = \overline C_0 \ssm \bigcup_{\{C \subset \overline C_0 \mid
\cod C - \cod C_0 \ge 2\}} \overline C.
\]
Let $p: \orb{\overline C_0} \to \Z$ be a Deligne--Bezrukavnikov perversity such that
\begin{equation}\label{eqn:pp-strict}
0 < p(C) - p(C_0) < \cod C - \cod C_0.
\qquad\text{for all $C \subset \overline C_0 \ssm \tilde C_0$.}
\end{equation}
Define two functions $\ufl p, \ush p: \orb{\overline C} \to \Z$ as
follows:
\[
\ufl p(C) =
\begin{cases}
p(C_0) & \text{if $C \subset \tilde C_0$,} \\
p(C) - 1 & \text{if $C \subset \overline C_0 \ssm \tilde C_0$,}
\end{cases}
\qquad
\ush p(C) =
\begin{cases}
p(C_0) & \text{if $C = C_0$,} \\
p(C_0) + 1 & \text{if $C \subset \tilde C_0 \ssm C_0$,} \\
p(C) + 1 & \text{if $C \subset \overline C_0 \ssm \tilde C_0$.}
\end{cases}
\]
It is easy to verify that $\ufl p$ and $\ush p$ are themselves monotone and comonotone Deligne--Bezrukavnikov perversities, so they give rise to additional pure-perverse $t$-structures on $\q\Dsp Xw$.  Note also that $\ufl p(C) \le p(C) \le \ush p(C)$ for all $C \subset \overline
C$.  (For $C \subset \tilde C_0 \ssm C_0$, this follows from the fact that
$0 \le p(C) - p(C_0) \le \cod C - \cod C_0 = 1$.)  Therefore, for any baric
perversity $q$, we have $\flatpq\Dlp X0w \subset \pq\Dlp X0w$ and
$\sharppq\Dgp X0w \subset \pq\Dgp X0w$.  Define full subcategories of
$\pq\cP(\tilde C_0)$ and of $\pq\cP(\overline C_0)$ as follows:
\begin{align*}
\pq\Pnp {\tilde C_0}w &= \flatpq\Dlp {\tilde C_0}0w \cap \sharppq\Dgp
{\tilde C_0}0w \\
\pq\Pnp {\overline C_0}w &= \flatpq\Dlp {\overline C_0}0w \cap \sharppq\Dgp
{\overline C_0}0w
\end{align*}

\begin{lem}\label{lem:cod1-pure}
Let $\cL$ be a sheaf in $\cg {C_0}$ that is $s$-pure of step $v \in \Z$. 
Define a Deligne--Bezrukavnikov perversity $p: \orb{\tilde C_0} \to \Z$
and a baric perversity $q: \orb{\tilde C_0} \to \Z$ by
\[
p(C) = r(C_0) - v
\qquad\text{and}\qquad
q(C) = \alt C_0 + 2\ufl r(C) - 2r(C_0).
\]
Let $w = 2v - \alt C_0$.  Then $\ru j_{!*}\cL[v-r(C_0)]|_{\tilde C_0} \in
\pq\Pnp {\tilde C_0}w$.
\end{lem}
\begin{proof}
Let $\cF = \ru j_{!*}(\cL[v-r(C_0)])|_{\tilde C_0}$.  We know that $\cF \in
\flatru\Dl {\tilde C_0}0$, so $\Tlt {r(C_0)-v}\cF$ belongs to $\flatru\Dl
{\tilde C_0}0$ as well.  Since $\cF|_{C_0} \cong \cL[v-r(C_0)]$, we see
that $\Tlt {r(C_0)-v}\cF$ is supported on $\tilde C_0 \ssm C_0$, so in fact
$\Tlt {r(C_0)-v}\cF \in \ru\Dl {\tilde C_0}{-1}$.  But there can be no
nonzero morphism from an object of $\ru\Dl {\tilde C_0}{-1}$ to one in
$\ru\cM(\tilde C_0)$, so $\Tlt {r(C_0)-v}\cF = 0$, and $\cF \in \Dg {\tilde
C_0}{r(C_0)-v}$.

Next, we have
\[
h^k(\cF) \in {}_{\sst 2\flat r}\cgl {\tilde C_0}{-2k} = \q\cgl {\tilde
C_0}{-2k + 2r(C_0) - \alt C_0}.
\]
We have just seen that $h^k(\cF) = 0$ for $k < r(C_0) - v$.  When $k \ge
r(C_0) - v$, we have
\[
-2k + 2r(C_0) - \alt C_0 \le 2v - \alt C_0 = w,
\]
and the inequality is strict when $k > r(C_0) - v$.  Thus, $\cF \in \q\Dsl
{\tilde C_0}w$, and $\Tgt {r(C_0)-v}\cF \in \q\Dslt {\tilde C_0}w$.  The
distinguished triangle
\[
\Tl {r(C_0)-v}\cF \to \cF \to \Tgt {r(C_0)-v}\cF \to
\]
then shows that $\cF \in \q\Dll {\tilde C_0}{r(C_0)-v}w = \flatpq\Dll
{\tilde C_0}0w$.

It remains to show that $\cF \in \sharppq\Dgg {\tilde C_0}0w$.  
Let $\cG = \D\cF$.  Then $\cG$ also arises as an intermediate-extension. 
Specifically, let $\cL' = (\D\cL)[\cod C_0]$; then $\cL'$ is a sheaf in
$\cg {C_0}$ that is pure of step $v' = \alt C_0 - v$.  We have $\cG = \bru
j_{!*}(\cL'[v' - \dualr(C_0)])|_{\tilde C_0}$.  By the arguments above, we
know that $\cG \in {}_{\sst \qprime}\Dll {\tilde C_0}{\dualr(C_0)-v'}{w'}$,
where
\[
\qprime(C) = \alt C_0 + 2\ufl\dualr(C) - 2\dualr(C_0)
\qquad\text{and}\qquad
w' = 2v' - \alt C_0.
\]
Observe that
\[
w' =  2(\alt C_0 - v) - \alt C_0 = \alt C_0 - 2v = -w.
\]
Next, note that $\cod C - \cod C_0 = \ush p(C) - \ufl p(C)$ for all $C
\subset \tilde C_0$, so
\begin{align*}
\ush p(C) &= \cod C - \cod C_0 + \ufl p(C) \\
&=\cod C - \cod C_0 + (r(C_0) - v) \\
&=\cod C - \cod C_0 + (\alt C_0 + \cod C_0 - \dualr(C_0) - (\alt C_0 - v'))
\\
&=\cod C - (\dualr(C_0) - v').
\end{align*}
It follows that $\D({}_{\sst \qprime}\Dll {\tilde C_0}{\dualr(C_0)-v'}{-w})
= {}^{\sst \ush p}_{\sst \dualqprime}\Dgg {\tilde C_0}0w$.  
From the formula
\[
\dualqprime(C) = 2\alt C - (\alt C_0 + 2\ufl\dualr(C) -
2\dualr(C_0)),
\]
we see that $\dualqprime(C_0) = \alt C_0 = q(C_0)$, and that for $C
\subset \tilde C_0 \ssm C_0$, we have
\begin{align*}
\dualqprime(C) 
&= 2\alt C - \alt C_0 - 2(\scod C - r(C) - 1) + 2(\scod C_0 - r(C_0)) \\
&= \alt C_0 - 2(\cod C - \cod C_0) + 2 + 2r(C) - 2r(C_0) > q(C).
\end{align*}
Thus, $\dualqprime(C) \ge q(C)$ for all $C$, so $\cF \cong \D\cG \in
\sharppq\Dgg {\tilde C_0}0w$, as desired.
\end{proof}

\begin{prop}\label{prop:purep-ic}
Let $\tj: \tilde C_0 \hto \overline C_0$ denote the inclusion.  Assume that $p: \orb X \to \Z$ satisfies condition~\eqref{eqn:pp-strict}.  Then
$\tj^*$ induces an equivalence of categories $\pq\Pnp {\overline C_0}w \to
\pq\Pnp {\tilde C_0}w$.
\end{prop}
\begin{proof}
The proof of this proposition is copied verbatim, except for minor changes
in notation, from~\cite[Proposition~2.3]{as:pcs}, which in turn is closely
based on~\cite[Theorem~2]{bez:pc}.
Let $J_{!*}: \q\Dsp {\overline C_0}w \to \q\Dsp {\overline C_0}w$ be the
functor
$\flatpq\Tl 0 \circ \sharppq\Tg 0$.  We claim that $J_{!*}$ actually
takes values in $\pq\Pnp {\overline C_0}w$.  Given $\cF \in
\q\Dsp {\overline C_0}w$, let $\cF_1 = \sharppq\Tg 0\cF$.  Then we have a
distinguished triangle 
\[
(\flatpq\Tg 1\cF_1)[-1] \to J_{!*}(\cF) \to \cF_1 \to.
\]
Note that $(\flatpq\Tg 1\cF_1)[-1] \in \flatpq\Dgp {\overline C_0}2w$. 
Now, $\ush p(C) - \ufl p(C) \le 2$ for all $C \subset \overline C_0$, and
this implies that $\flatpq\Dgp {\overline C_0}2w \subset \sharppq\Dgp
{\overline C_0}0w$.  Clearly, $\cF_1 \in \sharppq\Dgp {\overline C_0}0w$,
so it follows that $J_{!*}\cF \in \sharppq\Dgp {\overline C_0}0w$.  Since
$J_{!*}$ obviously takes values in $\flatpq\Dlp {\overline C_0}0w$, we have
$J_{!*}\cF \in \pq\Pnp {\overline C_0}w$.

Next, note that if $\cF \in \q\Dsp {\overline C_0}w$ is such that
$\cF|_{\tilde C_0} \in
\pq\Pnp {\tilde C_0}w$, then both $(\sharppq\Tg 0\cF)|_{\tilde C_0}$ and
$(\flatpq\Tl 0\cF)|_{\tilde C_0}$, and hence $(J_{!*}\cF)|_{\tilde C_0}$,
are isomorphic to $\cF|_{\tilde C_0}$.  In
particular, we can see now that $\tj^*$ is essentially surjective.  Given
$\cF
\in \pq\Pnp {\tilde C_0}w$, let $\tilde \cF$ be any object in $\Db
{\overline C_0}$
such that $\tj^*\tilde \cF \cong \cF$.  (Such an object exists
by~\cite[Corollary~2]{bez:pc}.)  Replacing $\tilde\cF$ by $\q\Bl w\q\Bg
w\tilde \cF$, we may assume that $\tilde \cF \in \q\Dsp {\overline C_0}w$. 
Then $\cF' = J_{!*}(\tilde\cF)$ is an object of $\pq\Pnp {\overline C_0}w$
such that $\tj^*\cF' \cong \cF$. 

Now, if $\phi: \cF \to \cG$ is a morphism in $\pq\Pnp {\tilde C_0}w$, then
by~\cite[Corollary~2]{bez:pc}, we can find objects $\cF'$ and $\cG'$ in
$\Db {\overline C_0}$ and a morphism $\phi': \cF' \to \cG'$ such that
$\tj^*\cF'
\cong \cF$, $\tj^*\cG' \cong \cG$, and $\tj^*\phi' \cong \phi$.  By
applying $\q\Bl w \circ \q\Bg w$ and then $J_{!*}$, we may assume that
$\cF'$, $\cG'$, and $\phi'$ actually belong to $\pq\Pnp {\overline C_0}w$. 
This shows that $\tj^*$ is full. 

To show that $\tj^*$ is faithful, it suffices to show that if $\phi$ is
an isomorphism, then $\phi'$ must be as well.  Since $\phi'|_{\tilde C_0}$
is an isomorphism, the kernel and cokernel of $\phi'$ must be supported on
$\overline C_0 \ssm \tilde C_0$.  Thus, the proof of the proposition will
be complete once we prove that an object of $\pq\Pnp {\overline C_0}w$ has
no nonzero subobjects or quotients in $\pq\Pp {\overline C_0}w$ that are
supported on $\overline C_0 \ssm \tilde C_0$.

Let $\cF \in \pq\Pnp {\overline C_0}w$, and let $\cG \in \pq\Pp {\overline
C_0}w$ be a nonzero object supported on $\overline C_0 \ssm \tilde C_0$. 
We will actually show that $\Hom(\cF,\cG) = \Hom(\cG,\cF) = 0$.  There
exists some closed subscheme structure $i: Z \hto \overline C_0$ on
$\overline C_0 \ssm \tilde C_0$ and some object $\cG' \in \pq\Pp Zw$ such
that $\cG \cong i_*\cG'$.  Then $\Hom(\cF,\cG) \cong \Hom(Li^*\cF, \cG')$. 
By Lemma~\ref{lem:pb-purep-exact}, $Li^*\cF \in \flatpq\Dmll Z0w$. 
Clearly, $\flatpq\Dmll Z0w = \pq\Dmll Z{-1}w$, and since $\cG' \in \pq\Dgg
Z0w$, we see that $\Hom(Li^*\cF,\cG') = 0$.  Similarly, $\Hom(\cG,\cF) =
\Hom(\cG', Ri^!\cF) = 0$ because $Ri^!\cF \in \sharppq\Dpgg Z0w = \pq\Dpgg
Z1w$.
\end{proof}

\begin{prop}\label{prop:ic-pure}
Let $\cL \in \cg {C_0}$ be a coherent sheaf, $s$-pure of step $v$.  Define a
Deligne--Bezrukavnikov perversity $p: \orb{\overline C_0} \to \Z$ and a
baric perversity $q: \orb{\overline C_0} \to \Z$ by
\begin{align*}
p(C) &= 
\begin{cases}
r(C_0) - v & \text{if $C \subset \tilde C_0$,} \\
r(C_0) - v + \cod C - \cod C_0 - 1 & \text{if $C \subset \overline C_0 \ssm
\tilde C_0$,}
\end{cases} \\
q(C)
&= 
\begin{cases}
\alt C_0 + 2r(C) - 2r(C_0) - 2\cod C + 2\cod C_0 & \text{if $C \subset
\tilde C_0$,} \\
\alt C_0 + 2r(C) - 2r(C_0) - 2\cod C + 2\cod C_0 + 1
& \text{if $C \subset \overline C_0 \ssm \tilde C_0$.}
\end{cases}
\end{align*}
Let $w = 2v - \alt C_0$.  Then $\ru j_{!*}(\cL[v-r(C_0)]) \in \pq\Pnp
{\overline C_0}w$.
\end{prop}
\begin{proof}
We first prove that $p$ is a monotone and comonotone Deligne--Bezrukavnikov perversity. 
Suppose $C' \subset \overline C$.  It is easy to check that
\[
p(C') - p(C) =
\begin{cases}
0 & \text{if $C, C' \subset \tilde C_0$,} \\
\cod C' - \cod C_0 - 1 & \text{if $C \subset \tilde C_0$ but $C'
\not\subset \tilde C_0$,} \\
\cod C' - \cod C & \text{if $C, C' \subset \overline C_0 \ssm \tilde C_0$.}
\end{cases}
\]
In all cases, it follows that
\[
0 \le p(C') - p(C) \le \cod C' - \cod C.
\]
Moreover, in the case where $C = C_0$ and $C' \subset \overline C_0 \ssm \tilde C_0$, we know that $\cod C' - \cod C_0 \ge 2$, and it follows that condition~\eqref{eqn:pp-strict} holds.

Note that for $C \subset \tilde C_0 \ssm C_0$, we have $\cod C - \cod C_0 =
1$, so the restrictions to $\orb{\tilde C_0}$ of the functions $p$ and $q$
defined here agree with those defined in Lemma~\ref{lem:cod1-pure}.  Let
$\cF = \ru j_{!*}(\cL[v - r(C_0)])$.  By Lemma~\ref{lem:cod1-pure},
$\cF|_{\tilde C_0} \in \pq\Pnp {\tilde C_0}w$.  Then, because the
inequalities~\eqref{eqn:pp-strict} hold, we may invoke
Proposition~\ref{prop:purep-ic}, which gives us a unique object $\cG \in
\pq\Pnp {\overline C_0}w$ such that $\tj^*\cG \cong \cF|_{\tilde C_0}$. 
We must show that $\cG \cong \cF$.  

A straightforward calculation shows that
\[
\ufl p(C) + \frac{q(C) + w}{2}
=
\begin{cases}
r(C_0) & \text{if $C = C_0$,} \\
r(C) - 1& \text{if $C \subset \tilde C_0 \ssm C_0$,} \\
r(C) - \frac{3}{2}& \text{if $C \subset \overline C_0 \ssm \tilde C_0$.}
\end{cases}
\]
Thus, $\ufl p(C) + \lceil \frac{q(C) + w}{2}\rceil = \ufl r(C)$.  Since
$\cG \in \flatpq\Pp {\overline C_0}w$, Proposition~\ref{prop:purep-stag}
tells us that $\cG \in \flatru\cM(\overline C_0)$.  Similarly, we have
\[
\ush p(C) + \frac{q(C) + w}{2}
=
\begin{cases}
r(C_0) & \text{if $C = C_0$,} \\
r(C) & \text{if $C \subset \tilde C_0 \ssm C_0$,} \\
r(C) + \half & \text{if $C \subset \overline C \ssm \tilde C_0$.}
\end{cases}
\]
Let $s(C) = \ush p(C) + \lceil \frac{q(C) + w}{2}\rceil$.  Then, as before,
Proposition~\ref{prop:purep-stag} tells us that $\cG \in {}^{\sst
s}\cM(\overline C_0)$.  But $s$ and $\ush r$ agree on $\orb{\overline
C_0} \ssm \orb{\tilde C_0}$, and we already know that $\cG|_{\tilde
C_0} \cong \cF|_{\tilde C0} \in \sharpru\cM(\tilde C_0)$, so we may
conclude that $\cG \in \sharpru\cM(\overline C_0)$.

Since $\cF$ is, up to isomorphism, the unique object in
$\flatru\cM(\overline C_0) \cap \sharpru\cM(\overline C_0)$ with the
property that $\cF|_{C_0} \cong \cL[v-r(C_0)]$, we conclude that $\cG \cong
\cF$, as desired.
\end{proof}

The formulas for the perversities used in Proposition~\ref{prop:ic-pure}
are carefully chosen so as to ensure that, after calculating $\ufl p(C) +
\frac{q(C)+w}{2}$ and $\ush p(C) + \frac{q(C)+w}{2}$, we are able to invoke
Proposition~\ref{prop:purep-stag}.  Unfortunately, those calculations have
the aesthetically unpleasant property of not being integer-valued.  We
could perhaps improve the aesthetics by modifying the definition of $q$. 

Let us briefly study how this would change the subsequent calculations. 
We retain all the notation used in the proof of
Proposition~\ref{prop:ic-pure}, including the definition of $q$.  
We have proved that $\cF \in \q\Dsl {\overline C_0}w$, or, equivalently,
that
\begin{equation}\label{eqn:ic-pure1}
i_C^*h^k(\cF)|_C \in \cgl C{\lfloor (w+q(C))/2\rfloor}
\end{equation}
for all $k$.  Note that $w \equiv \alt C_0 \pmod{2}$.  From the definition
of $q$, we see that
\[
\begin{aligned}
q(C) + w &\equiv 0 \pmod 2 &&\text{if $C \subset \tilde C_0$,} \\
q(C) + w &\equiv 1 \pmod 2 &&\text{if $C \subset \overline C_0 \ssm \tilde
C_0$.}
\end{aligned}
\]
For $n \equiv 1 \pmod 2$, we have $\lfloor n/2\rfloor = (n-1)/2$, 
so we can refine~\eqref{eqn:ic-pure1} by defining $\qprime: \orb{\overline
C_0} \to \Z$ by
\[
\qprime(C) = \alt C_0 + 2r(C) - 2r(C_0) - 2\cod C + 2\cod C_0
=
\begin{cases}
q(C) & \text{if $C \subset \tilde C_0$,} \\
q(C) - 1 & \text{if $C \subset \overline C_0 \ssm \tilde C_0$.}
\end{cases}
\]
We then have
\[
i_C^*h^k(\cF)|_C \in \cgl C{(w+\qprime(C))/2},
\]
so $\cF \in {}_{\sst \qprime}\Dsl {\overline C_0}w$.  

By replacing $q$ by $\qprime$, we have lost the two-sided nature of
Proposition~\ref{prop:ic-pure}: it is not true in general that $\cF \in
{}_{\sst \qprime}\Dsg {\overline C_0}w$.  For a one-sided statement alone,
however, we could further replace $\qprime$ by any larger function. 
Pushing forward to $\Db X$ by $i_{C_0*}$, we obtain the following useful
result.

\begin{cor}\label{cor:ic-bound}
Let $\cL \in \cg {C_0}$ be a coherent sheaf, $s$-pure of step $v$.  Let
$q: \orb X \to \Z$ be any baric perversity such that
\[
q(C) \ge \alt C_0 + 2r(C) - 2r(C_0) - 2\cod C + 2\cod C_0
\qquad\text{if $C \subset \overline C_0$,}
\]
and let $w = 2v - \alt C_0$.  Then $\ru\cIC(\overline C_0, \cL[v - r(C_0)])
\in \q\Dsl Xw$. \qed
\end{cor}
Note that no conditions are imposed on the values of $q(C)$ for $C
\not\subset \overline C_0$.  Since $\ru\cIC(\overline C_0, \cL[v -
r(C_0)])$ is supported on $\overline C_0$, it is clear that the values of
$q$ outside $\overline C_0$ have no bearing on this statement.

Recall that a simple staggered sheaf $\cF = \ru\cIC(\overline C_0, \cL[v-r(C_0)])$ is characterized by the property that $Li_C^*\cF \in \ru\Dmlt {\overline C}0$ and $Ri_C^!\cF \in \ru\Dpgt {\overline C}0$ for all $C \subset \overline C_0 \ssm C_0$.  The following result, which illustrates the use of Corollary~\ref{cor:ic-bound}, gives a baric analogue of this property in the case of the self-dual staggered perversity.  (This result will not be used in the sequel.)  

\begin{prop}\label{prop:ic-res-strict}
Assume that $r(C) = \half\scod C$.  Let $\cL \in \cg {C_0}$ be a coherent sheaf, $s$-pure of step $v$, and let $w = 2v - \alt C_0$.  For any orbit $C \subset \overline C_0 \ssm C_0$, we have
$Li_C^*\ru\cIC(\overline C_0, \cL[v-\half\scod C_0]) \in \Dmslt {\overline C}w$ and
$Ri_C^!\ru\cIC(\overline C_0, \cL[v-\half\scod C_0]) \in \Dpsgt {\overline C}w$.
\end{prop}
\begin{proof}
Consider the baric perversity $q: \orb X \to \Z$ given by
\[
q(C) = \begin{cases}
\alt C & \text{if $C \not\subset \overline C_0 \ssm C_0$,} \\
\alt C - 1 & \text{if $C \subset \overline C_0 \ssm C_0$.}
\end{cases}
\]
This function obeys the condition in Corollary~\ref{cor:ic-bound} with respect to the middle staggered perversity $r(C) = \half\scod C$:
\[
q(C) \ge \alt C_0 + \scod C - \scod C_0 - 2\cod C + 2\cod C_0 = \alt C + \cod C_0 - \cod C
\]
for all $C \subset \overline C_0$, since $\cod C_0 - \cod C \le -1$ for any $C \subset \overline C_0 \ssm C_0$.  Invoking that corollary, we have $\cIC(\overline C_0, \cL[v - \half\scod C_0]) \in \q\Dsl Xw$.  It
follows that $Li_C^*\cIC(\overline C_0, \cL[v-\half\scod C_0]) \in \q\Dmsl {\overline C}w$ by Lemma~\ref{lem:baric-res}.  Since $q(C') = \alt C' - 1$ for all $C' \in \orb{\overline C}$, it
follows that $Li_C^*\cIC(\overline C_0,
\cL[v - \half\scod C_0]) \in \Dmslt {\overline C}w$.

The same argument applies to
$\D\cIC(\overline C_0, \cL[v-\half\scod C_0]) \cong \cIC(\overline C_0,
\D(\cL[v - \half\scod C_0]))$, and shows that $Li_C^*\D\cIC(\overline C_0, \cL[v
- \half\scod C_0]) \in \Dmslt \overline{C}{-w}$.  Since $Ri_C^!\cIC(\overline C_0, \cL[v -\half\scod C_0]) \cong \D(Li_C^*\D\cIC(\overline C_0, \cL[v - \half\scod C_0]))$, we
conclude that $Ri_C^!\cIC(\overline C_0, \cL[v - \half\scod C_0]) \in \Dpsgt
{\overline C}w$, as desired.
\end{proof}

\section{The Baric Purity Theorem}
\label{sect:purity}

In this section, we prove the baric version of the Purity Theorem for staggered sheaves.   
Henceforth, unless otherwise specified, all references to baric degrees,
purity, and baric truncation should be understood to be with respect to the
self-dual baric structure $(\{\Dsl Xw\},\{\Dsg Xw\})_{w \in \Z}$
corresponding to the middle baric perversity $q(C) = \alt C$.  In
particular, the left-subscript ``$q$'' will generally be omitted.

\begin{defn}
A staggered perversity $r: \orb X \to \Z$ is said to be \emph{moderate} if
for any two orbits $C, C' \subset X$ with $C' \subset \overline C$, the
following inequalities all hold:
\begin{gather}
\cod C' - \cod C \le r(C') - r(C) \le \alt C' - \alt C 
\label{eqn:f1}\\
\half\alt C' - \half\alt C \le r(C') - r(C) \le \half\alt C' + \cod C' -
\half\alt C - \cod C \label{eqn:f2}
\end{gather}
\end{defn}

\begin{rmk}
Note that a necessary condition for the existence of a moderate staggered
perversity is that
\[
\cod C' - \cod C \le \alt C' - \alt C
\]
whenever $C' \subset \overline C$.  Under these conditions, the staggered perversities $r(C) = \lfloor \half\scod C\rfloor$ and $r(C) = \lceil \half\scod C \rceil$ are automatically moderate.
\end{rmk}

\begin{lem}\label{lem:ic-pure}
Let $\cL \in \cg {C_0}$ be a coherent sheaf, $s$-pure of step $v$.  If $r$
is a moderate staggered perversity, $\ru\cIC(\overline C_0, \cL[v-r(C_0)])$
is pure of baric degree $w = 2v - \alt C_0$.
\end{lem}
\begin{proof}
Let $\cF = \ru\cIC(\overline C_0, \cL[v-r(C_0)])$.
It follows from the inequalities~\eqref{eqn:f2} that
\[
\alt C \ge \alt C_0 + 2r(C) - 2r(C_0) - 2\cod C + 2\cod C_0
\]
for all $C \subset \overline C_0$.  Then Corollary~\ref{cor:ic-bound}
tells us that $\cF \in \Dsl Xw$.  Note that the dual of a moderate perversity is also moderate, so we may
apply the same argument to $\D\cF \in \bru \cM(X)$.  We find that $\D\cF \in \Dsl X{-w}$, so $\cF$ is pure of baric degree $w$.
\end{proof}

\begin{prop}\label{prop:bary-trunc}
Let $r: \orb X
\to \Z$ be a moderate staggered perversity.  Then the category of staggered
sheaves
$\ru\cM(X)$ is stable under the baric truncation functors $\Bl w$ and $\Bg
w$ with respect to the middle baric perversity.
\end{prop}
\begin{proof}
Since every staggered sheaf has finite length, we may proceed by induction
on the length of $\cF$. If $\cF$ is simple, Lemma~\ref{lem:ic-pure} tells
us that $\cF$ is pure.  In particular, every baric truncation functor takes
$\cF$ either to itself or to $0$.

Now, suppose $\cF$ is not simple.  Let $\cF' \subset \cF$ be a simple
subobject, and form a short exact sequence
\[
0 \to \cF' \to \cF \to \cF'' \to 0.
\]
For any $w \in \Z$, we obtain a distinguished triangle
\[
\Bl w \cF' \to \Bl w\cF \to \Bl w\cF'' \to.
\]
The first term is in $\ru\cM(X)$ because $\cF'$ is simple, and the last
term is in $\ru\cM(X)$ by induction.  Therefore, $\Bl w\cF \in \ru\cM(X)$
as well.  The same argument shows that $\ru\cM(X)$ is stable under $\Bg w$
as well.
\end{proof}

Below is the first major theorem of the paper.  The parts of this theorem correspond to Proposition~5.3.1, Corollaire~5.3.4, Th\'eor\`eme~5.3.5, and Th\'eor\`eme~5.4.1 in~\cite{bbd}, respectively.

\begin{thm}[Baric Purity]\label{thm:purity}
Suppose $X$ is endowed with a recessed $s$-structure.
Let $r: \orb X \to \Z$ be a moderate staggered perversity.
\begin{enumerate}
\item Let $\cF$ be a staggered
sheaf.  If $\cF \in \Dsl Xw$, then every
subquotient of $\cF$ is in $\Dsl Xw$.  If $\cF \in \Dsg Xw$, then every
subquotient of $\cF$ is in $\Dsg Xw$.
\item Every simple staggered sheaf
is pure.
\item Every staggered sheaf
$\cF$ admits a unique finite filtration
\[
\cdots \subset \cF_{\le w-1} \subset \cF_{\le w} \subset \cF_{\le w+1}
\subset \cdots
\]
such that $\cF_{\le w}/\cF_{\le w-1} \in \Dsp Xw$.
\item Let $\cF \in \Db X$.  Then
$\cF \in \Dsl Xw$ if and only if $\ru
h^i(\cF) \in \Dsl Xw$ for all $i$, and $\cF \in \Dsg Xw$ if and only if
$\ru h^i(\cF) \in \Dsg Xw$ for all $i$.
\end{enumerate}
\end{thm}
\begin{proof}
(1)~Suppose we have a short exact sequence of staggered sheaves $0 \to
\cF' \to \cF \to \cF'' \to 0$, with $\cF \in \Dsl Xw$.  Applying the
functor $\Bgt w$ to this sequence yields a new short exact sequence in
$\ru \cM(X)$ with middle term $0$.  Therefore, $\Bgt w\cF' = \Bgt w\cF'' =
0$ as well.  The proof for $\cF \in \Dsg Xw$ is similar.

(2)~This was proved in Lemma~\ref{lem:ic-pure}.

(3)~The desired filtration is given by $\cF_w = \Bl w\cF$.

(4)~If all $\ru h^i(\cF) \in \Dsl Xw$, the fact that $\cF \in \Dsl Xw$
follows (by induction on the number of nonzero cohomology objects) from
the fact that $\Dsl Xw$ is stable under extensions.  Conversely, suppose
$\cF \in \Dsl Xw$.  We proceed by induction on the number of nonzero
cohomology objects.  If $\cF$ has only one nonzero cohomology object,
there is nothing to prove.  Otherwise, choose some $k$ such that $\ru\Tl
k\cF$ and $\ru\Tg {k+1}\cF$ are both nonzero.  By
Proposition~\ref{prop:bary-trunc}, $\Bgt w\ru\Tg {k+1}\cF \in \ru\Dg
X{k+1}$, so
\[
\Hom(\ru\Tg {k+1}\cF, \Bgt w\ru\Tg {k+1}\cF) \cong \Hom(\cF, \Bgt w\ru\Tg
{k+1}\cF) = 0,
\]
where the last equality holds because $\cF \in \Dsl Xw$.  It follows that
$\Bgt w\ru\Tg {k+1}\cF = 0$, so $\ru\Tg {k+1}\cF \in \Dsl Xw$, and hence
$\ru\Tl k\cF \in \Dsl Xw$ as well.  By induction, we know that all
cohomology objects of $\ru\Tg {k+1}\cF$ and of $\ru\Tl k\cF$ lie in $\Dsl
Xw$, so all $\ru h^i(\cF) \in \Dsl Xw$.
\end{proof}

\section{$s$-structures on a $G$-orbit}
\label{sect:orbit}

In this section only, we assume that the ground field $\Bbbk$ is
algebraically closed.  

Let $C \subset X$ be a $G$-orbit.  Our goal in this section is to classify $s$-structures on $C$ in terms of the representation theory of a certain algebraic torus $T_C$, defined as follows.  Choose a closed point $x \in C$, and let $H
\subset G$ be the stabilizer of $x$.  We assume throughout this section that $H$ is connected.  Let $R \subset H$ be the radical of $H$, 
and let $U \subset H$ be the unipotent radical of $H$.  Let $T_C$ be a
maximal torus of $R$.

We claim that $T_C$ is canonical: that is, that making different choices in
the preceding paragraph would lead to a torus canonically isomorphic to
$T_C$.  Let $x'$ be another closed point of $C$, with stabilizer $H'$, and
let $T_C'$ be a maximal torus in the radical $R'$ of $H'$.  There is some
$g \in G$ such that $g \cdot x = x'$.  Then $g H g^{-1} = H'$, and $g T_C
g^{-1}$ is another maximal torus in $R'$.  Any two maximal tori in $R'$ are
conjugate, so by replacing $g$ by $r'g$ for a suitable $r' \in R'$, we may
achieve that $g T_C g^{-1} = T_C'$.  We thus obtain an isomorphism $f: T_C
\overset{\sim}{\to} T_C'$ given by $f(t) = gtg^{-1}$.   To show that $f$ is
independent of $g$, suppose $g' \in G$ is another element such that
$g'\cdot x = x'$ and $g'T_C(g')^{-1} = T_C'$.  Then $g' = gh$, where $h \in
H$ normalizes $T_C$.  But then it follows that $h$ centralizes $T_C$: the
image of $T_C$ in the reductive group $H/U$ is central, so for any $t \in
T_C$, we have $hth^{-1} = tu$ for some $u \in U$, and $tu \in T_C$ implies $u
= 1$.  We conclude that the isomorphism $T_C \cong T_C'$ given by
conjugation by $t \mapsto g't(g')^{-1}$ coincides with $f$.

Next, let $\cO_H$ and $\cO_U$ denote the $\Bbbk$-algebras of regular
functions on $H$ and $U$ respectively.  We will regard them as $H$-modules
and in particular as $T_C$-modules via the action
\[
g \cdot f : h \mapsto f(g^{-1} h g).
\]
Let $X(T_C)$ and $Y(T_C)$ denote the character and cocharacter lattices of
$T_C$, respectively.  Let $\fu$ denote the Lie algebra of $U$, and define
a subset $\Upsilon_C$ by
\[
\Upsilon_C = \{ \upsilon \in X(T_C) \mid
\text{$\upsilon$ occurs in the adjoint action of $T_C$ on $\fu$} \}.
\]
Let $S(\fu^*)$ denote the symmetric algebra on the dual vector space to
$\fu$.  In other words, $S(\fu^*)$ is the ring of regular functions $\fu
\to \Bbbk$.  Next, let $-\N\Upsilon_C$ denote the set of all nonpositive
integer linear combinations of elements of $\Upsilon_C$.  Clearly, the set
of $T_C$-weights on $\fu^*$ is $-\Upsilon_C$, and the set of $T_C$-weights
on $S(\fu^*)$ is $-\N\Upsilon_C$.  We will see later that $-\N\Upsilon_C$
is also the set of $T_C$-weights on $\cO_H$ and on $\cO_U$.

\begin{prop}\label{prop:sstruc-cc}
There is a canonical injective map
\[
\Psi_C: \{\text{$s$-structures on $C$}\} \to Y(T_C).
\]
\end{prop}
\begin{proof}
We retain the notation used above: $H$ is the stabilizer of some closed
point $x \in C$, $R$ and $U$ are its radical and unipotent radical, and
$T_C$ is a maximal torus in $R$.  Recall that the category $\cg C$ of
$G$-equivariant coherent sheaves on $C$ is equivalent to the category $\rg
H$ of finite-dimensional algebraic representations of $H$.  Moreover, this
equivalence respects tensor products and internal Hom.  For the remainder
of the proof, we will work exclusively in the setting of $\rg H$.  In
particular, an ``$s$-structure'' will now mean a collection of full
subcategories $(\{\rgl Hw\}, \{\rgg Hw\})_{w \in \Z}$, subject to various
axioms.

Consider the reductive group $M = H/U$.  We identify $T_C$ with its image
in $M$, {\it viz.}, the identity component of the center of $M$.  $U$ acts
trivially in any irreducible representation of $H$, so the simple objects
in $\rg H$ can be identified with the irreducible representations of $M$. 
In any $s$-structure, every simple object is $s$-pure of some step.  Since
the categories $\rgl Hw$ and $\rgg Hw$ are stable under extensions, the
entire $s$-structure is determined by the steps of simple objects.

Consider first the set of $1$-dimensional representations of $H$.  The
semisimple group $H/R \cong M/T_C$ has a unique $1$-dimensional
representation (the trivial one), so in any two nonisomorphic
$1$-dimensional representations of $M$, $T_C$ must act by distinct
characters.  Thus, the set of $1$-dimensional representations can be
identified with a sublattice of $X(T_C)$, which we will denote $X(M)$.  We
claim that for any $\lambda \in X(T_C)$, some multiple of $\lambda$ lies in
$X(M)$.  There certainly exists some extension of $\lambda$ to a character
of a maximal torus of $M$, and hence there is some irreducible
$M$-representation $V$ in which $T_C$ acts by $\lambda$.  The top exterior
power $\bigwedge^{\dim V} V$ is a $1$-dimensional $M$-representation
contained in $\bigotimes^{\dim V} V$, so $T_C$ acts on it by the character
$(\dim V)\lambda$.  Thus, $(\dim V)\lambda \in X(M)$.

Given an $s$-structure, define a function $\phi: X(T_C) \to \Q$ by putting
\[
\phi(\lambda) = \step \lambda \qquad\text{for all $\lambda \in X(M)$.}
\]
Note that it suffices to define $\phi$ on $X(M)$ because every character in $X(T_C)$ has some multiple in $X(M)$.  Now, for any irreducible $M$-representation
$V$, if $\lambda \in X(T_C)$ is the character of $T_C$ on $V$, we have
$\step \bigwedge^{\dim V} V = (\dim V)\step V$ and $\phi(\bigwedge^{\dim V}
V) = (\dim V)\phi(\lambda)$, so it follows that
\begin{equation}\label{eqn:sstruc-cc}
\phi(\lambda) = \step V 
\qquad\text{if $\lambda$ is the character of $S$ on $V$.}
\end{equation}
In particular, $\phi$ takes values in $\Z$, so we may regard it as an element of $Y(T_C)$.  Since any $s$-structure is determined by the steps of simple objects, it is
clear from~\eqref{eqn:sstruc-cc} that distinct $s$-structures give rise to
distinct cocharacters $\phi \in Y(T_C)$.
\end{proof}

We can describe the image of $\Psi_C$ quite precisely.

\begin{defn}
A cocharacter $\phi \in Y(T_C)$ is said to be \emph{semifocused} if
$\phi(\upsilon) \le 0$ for all $\upsilon \in \Upsilon_C$.  It is
\emph{focused} if $\phi(\upsilon) < 0$ for all $\upsilon \in \Upsilon_C$.
\end{defn}

\begin{thm}
\label{thm:ssogo}
If $X$ and $G$ are schemes over an algebraically closed field, then a
cocharacter $\phi \in Y(T_C)$ is in the image of $\Psi_C$ if and only if it
is semifocused.
\end{thm}

Before proving this theorem, we need the following basic result.

\begin{lem}\label{lem:levi}
Let $K \subset H$ be a $T_C$-stable subgroup containing $U$.  Then there
is a $T_C$-equivariant isomorphism of varieties $K \cong K/U \times \fu$.
\end{lem}

Note that in characteristic $0$, this lemma is straightforward: $K$ admits
a Levi decomposition $K \cong K/U \ltimes U$, and the exponential map
provides a $T_C$-equivariant isomorphism of varieties $\fu \to U$. 
Neither Levi decompositions nor the exponential map necessarily exist in
positive characteristic, however.

\begin{proof}
The structure theory of unipotent groups provides a filtration
\begin{equation}\label{eqn:unip-filt}
1 = U_0 \subset U_1 \subset \cdots \subset U_n = U
\end{equation}
with the following properties: (1) each $U_i$ is a normal subgroup of $H$,
and therefore of $K$ and of $U_{i+1}$; (2) each $U_i$ is stable under the
action of
$T_C$; and (3) each subquotient $U_{i+1}/U_i$ is isomorphic to $\bbG_a$. 
Note that as a consequence of (1), each of the schemes $K/U_i$ is affine.

Let us show that each projection $K/U_{i-1} \to K/U_i$ admits a
$T_C$-equivariant section.  It is convenient to use the language of
algebraic stacks: put $X_i = K/U_i$ and let $[X_i/T_C]$ denote the quotient
stack.  The map $[X_{i-1}/T_C] \to [X_i/T_C]$ is a $\bbG_a \cong
U_i/U_{i-1}$-torsor over $[X_i/T_C]$ in the flat topology.  To show that it
has a section it suffices to show that $H^1_{\text{flat}}([X_i/T_C];\bbG_a)
= 0$.  Note that because $\bbG_a$ is commutative, we have access to higher
cohomology groups and the machinery of spectral sequences.  In particular,
associated to the composition of maps
\[
[X_i/T_C] \to [\mathit{pt}/T_C] \to \mathit{pt},
\]
there is the Leray spectral sequence
\[
E_2^{pq} = H^p([\mathit{pt}/T_C];H^q_{\text{flat}}(X_i;\bbG_a))\implies
H^{p+q}_{\text{flat}}([X_i/T_C];\bbG_a).
\]
We have $H^q_{\text{flat}}(X_i;\bbG_a) \cong
H^q_{\text{Zar}}(X_i;\cO_{X_i})$, which vanishes for $q>0$ because $X_i$ is
affine.  Moreover, because the category of $T_C$-representations is
semisimple, the cohomology groups $H^p([\mathit{pt}/T_C];\cF)$ vanish for
$p>0$ and any coherent sheaf $\cF$ on the classifying stack
$[\mathit{pt}/T_C]$.  Thus we have the required vanishing of
$H^1_{\text{flat}}([X_i/T_C];\bbG_a)$, and every map $K/U_{i-1} \to K/U_i$
has a $T_C$-equivariant section.

It follows that there is a $T_C$-equivariant isomorphism $K/U_{i-1} \cong
K/U_i \times U_i/U_{i-1}$.  Now, consider the Lie algebra version of the
filtration~\eqref{eqn:unip-filt}:
\[
0 = \fu_0 \subset \fu_1 \subset \cdots \subset \fu_n = \fu.
\]
Each quotient $\fu_i/\fu_{i-1}$ may be identified with the Lie algebra of
$U_i/U_{i-1}$.  The exponential map makes sense for $\bbG_a$ in arbitrary
characteristic, and provides a $T_C$-equivariant isomorphism
$\fu_i/\fu_{i-1} \to U_i/U_{i-1}$.  Combining the isomorphisms
$K/U_{i-1} \cong K/U_i \times \fu_i/\fu_{i-1}$ for all $i$, we obtain
\[
K \cong K/U \times \fu_1 \times \fu_2/\fu_1 \times\cdots \times
\fu/\fu_{n-1}.
\]
Again using the fact that $T_C$-representations are semisimple, we see
that there is a $T_C$-equivariant isomorphism $\fu_1 \times \fu_2/\fu_1
\times\cdots \times
\fu/\fu_{n-1} \cong \fu$, as desired.
\end{proof}

Applying this lemma in the special cases $K = U$ and $K = H$, we obtain
the following result.

\begin{cor}\label{cor:hopf}
We have isomorphisms of $T_C$-representations $S(\fu^*) \cong \cO_U$ and
$\cO_H \cong \cO_{H/U} \otimes S(\fu^*)$.\qed
\end{cor}

Since $T_C$ acts trivially on $\cO_{H/U}$, the last part of this corollary
implies that $T_C$ acts with the same set of weights on $\cO_H$ and on
$S(\fu^*)$.

\begin{proof}[Proof of Theorem~\ref{thm:ssogo}]
Let $M$ be an $H$-module.  Note that the comodule structure map $\gamma_M: M \to M \otimes \cO_H$ is $H$-equivariant.  (This is easiest to see by identifying $M \otimes \cO_H$ with the vector space of regular functions $H \to M$.)  
In particular, this map is $T_C$-equivariant and preserves weights.

Define $\sigma_{\leq w} M$ to be the vector subspace of $M$ spanned by
those weight vectors whose wieght $\chi$ satisfies $\phi(\chi) \leq w$.  To say that $\phi$ defines an $s$-structure is
equivalent to saying $\sigma_{\leq w} M$ is an $H$-submodule of $M$.  If $m
\in M$ has weight $\chi$, then we may write $\gamma_M(m)$ as $\sum m_i
\otimes f_i$ where $f_i$ has weight $-\upsilon_i$ for some $\upsilon_i \in
\Upsilon$ and $m_i$ has weight $\chi +\upsilon_i$.  Thus,
$\gamma_M(\sigma_{\leq w} M) \subset \sigma_{\leq w} M \otimes \cO_H$ if and only if
$\phi(\upsilon) \leq 0$ for all $\upsilon \in \Upsilon$.
\end{proof}

We conclude this section with an Ext-vanishing result for certain
$s$-structures.

\begin{thm}\label{thm:char0-sstruc}
Suppose that $H$ has a Levi factor $M$ and that the category of
$M$-representations is semisimple.  Let $\phi$ be a semifocused
cocharacter, and let $(\{\cgl Cw\}, \{\cgg Cw\})_{w \in \Z}$ be the
corresponding $s$-structure.  The following conditions are equivalent:
\begin{enumerate}
\item The cocharacter $\phi$ is focused.
\item For any two simple objects $\cF, \cG \in \cg C$ that are both
$s$-pure of step $w$, we have $\Ext^1(\cF,\cG) = 0$.
\end{enumerate}
\end{thm}
Note that the conditions on $H$ always hold in characteristic $0$, and
they always hold in arbitrary characteristic when $H$ is solvable.
\begin{proof}
Suppose $\rg H$ carries an $s$-structure such that the
corresponding cocharacter $\phi$ is focused.  Let $V_1, V_2 \in \rg H$ be
simple objects that are both $s$-pure of step $w$.  Suppose we have a short
exact sequence
\begin{equation}\label{eqn:foc-ses}
0 \to V_2 \to V \to V_1 \to 0.
\end{equation}
As a sequence of $M$-representations, this sequence splits, and we can find
a subspace $V_1' \subset V$ that is isomorphic as an $M$-representation to
$V_1$.  Let us show that $V_1'$ is an $H$-submodule of $V$.  It suffices to
show that $V_1'$ is stable under multiplication by $U$; as in the proof of
Theorem~\ref{thm:ssogo}, this follows from the fact that the comodule
structure map $\gamma_V: V \to V \otimes \cO_U$
preserves weights.  Indeed, if $v \in V_1'$, then write $\gamma_V(v)$ as
$\sum v_i \otimes u_i$, where $v_i \in V$ and $u_i \in \cO_U$
are weight vectors for $T_C$ of weights $\chi_i$ and $\upsilon_i$,
respectively.  Since $\phi$ is focused, and $\upsilon_i \in -\N\Upsilon_C
$, we must have $\phi(\chi_i) < w$
unless $\upsilon_i = 0$.  The latter condition only holds when $u_i$ is a
constant; thus $v_i$ cannot lie in $V_2$. Thus the
sequence~\eqref{eqn:foc-ses} splits.

Conversely, suppose $\phi$ is semifocused but not focused.  Then
$\Sg 0\fu \ne 0$.  By a slight abuse of notation, let us denote by
$\Sl 0\cO_H$ and $\Sl 0S(\fu^*)$ the subspaces of $\cO_H$
and $S(\fu^*)$, respectively, spanned by all $T_C$-weight spaces whose
weight $\chi$ satisfies $\phi(\chi) \le 0$.  (This notation is an abuse
because $\cO_H$ and $S(\fu^*)$ are infinite-dimensional and therefore not
objects of $\rg H$.)  It follows from Corollary~\ref{cor:hopf} that
$\Sl 0\cO_H \cong \cO_{H/U} \otimes \Sl 0S(\fu^*)$.  
Now, identify the dual space $(\Sg 0\fu)^*$ with a subspace of
$\fu^*$.  Since the weights occuring in $S(\fu^*)$ are linear
combinations with nonpositive coefficients of the weights in $\Upsilon_C$,
it is easy to see that $\Sl 0S(\fu^*) \cong S((\Sg 0\fu)^*)$.

We claim that $\Sl 0\cO_H$ is a Hopf subalgebra of $\cO_H$. 
Indeed, the multiplication map $\cO_H \otimes \cO_H \to \cO_H$, the
comultiplication map $\cO_H \to \cO_H \otimes \cO_H$, and the antipode
(inverse) map $\cO_H \to \cO_H$ are all $H$- and therefore
$T_C$-equivariant, so the restrictions of these maps to $\Sl 0\cO_H$
endow that space with the structure of a Hopf algebra.  Thus, $H' = \Spec
\Sl 0\cO_H$ is an affine algebraic group over $\Bbbk$, and the
inclusion $\Sl 0\cO_H \hto \cO_H$ corresponds to a surjective
group homomorphism $H \to H'$.  Note that $H'$ cannot be reductive: the
largest reductive quotient of $H$ is $H/U$, but since
$\cO_{H/U}$ can be identified with a subalgebra of $\Sl 0\cO_H$,
the group $H/U$ is a nontrivial quotient of $H'$.

Let $U'$ be the unipotent radical of $H'$.  Because the quotient map $H
\cong M \ltimes U \to H/U \cong M$ factors through $H'$, the latter group
inherits a Levi decomposition with the same Levi factor: we have $H' \cong
M \ltimes U'$.  Now, find a faithful representation
of $H'$ on some vector
space $V$.  Such a representation is not semisimple: the space $V^{U'}$ of
$U'$-fixed vectors (which is not all of $V$ because $U' \ne 1$) is an
$H'$-invariant subspace with no $H'$-invariant complement.  $V^{U'}$ does,
of course, admit an $M$-invariant complement; let $V_1$ be an irreducible
$M$-representation in that complement, and suppose it is $s$-pure of step
$w$.
Let $V'$ be the smallest $H$-stable subspace containing $V_1$, and find a
filtration
\[
0 = W_0 \subset W_1 \subset \cdots \subset W_n = V'
\]
such that $W_i/W_{i-1}$ is simple for each $i$.  Since $V_1$ is not
contained in any proper submodule of $V'$, we must have $W_n/W_{n-1}
\cong V_1$.  Moreover, since $V' \ne V_1$, we know that $n \ge 2$.  Let $W
= V'/W_{n-2}$, and let $W' = W_{n-1}/W_{n-2} \subset W$.  We then have a
short exact sequence
\[
0 \to W' \to W \to V_1 \to 0.
\]
This sequence cannot split: if $W$ contained an $H$-stable subspace
isomorphic to $V_1$, its preimage in $V'$ would be a proper
$H$-stable subspace of $V'$ containing $V_1$.  Thus, $\Ext^1(V_1, W') \ne
0$.  To finish the proof of the theorem, it remains only to show that
$\step W' = w$.  As usual, there is an $M$-stable subspace $V'_1
\subset W$ that is isomorphic to $V_1$ as an $M$-representation.  Moreover,
there is some vector $v \in V'_1$ whose
image under the comodule map $\gamma_W: W \to W \otimes \cO_{H'}$ is not
contained in $V'_1 \otimes \cO_{H'}$.  That is, if we write $\gamma_W(v)$
in the form $\sum v_i \otimes u_i$, where all the $v_i \in W$ and all the
$u_i \in \cO_{H'}$ are weight vectors, say of weights $\chi_i$ and
$\upsilon_i$, respectively, there is at least one nonzero term
with $v_i \notin V_1$, and therefore $v_i \in W'$.  Now, $\phi(\upsilon_i)
= 0$ by the construction of $H'$, so it follows that $\phi(v_i) = w$, and
hence that $\step W' = w$.  Thus, we have exhibited a pair of simple
objects $V_1, W'
\in \rg H$, both $s$-pure of step $w$, such that $\Ext^1(V_1,W') \ne 0$.
\end{proof}

\section{Higher Ext-Vanishing over a Closed Orbit}
\label{sect:ext-van}

Consider the following condition on an $s$-structure:

\begin{defn}
An $s$-structure is \emph{split} if for every orbit $C \in \orb X$, and any two simple objects $\cF, \cG \in \cg C$ that are both $s$-pure of step $v$, we have $\Ext^1(\cF,\cG) = 0$.
\end{defn}

For the remainder of the paper, we assume that the fixed $s$-structure on $X$ is both recessed and split.  Theorem~\ref{thm:char0-sstruc} gives a useful criterion for an $s$-structure to be split.

For a closed subspace $Z \subset X$, define
\[
\csupp XZ_{\ge w} = \{ \cF \in \cgg Xw \mid \text{$\cF$ is supported set-theoretically on $Z$} \}.
\]
The main result of this section is the following Ext-vanishing result, which will be an important tool in the proofs of both decomposition theorems.

\begin{prop}\label{prop:ext-closed}
Let $C \subset X$ be a closed orbit, and let $\cF \in \cg X$ be such that $i_C^*\cF \in \cgl Cw$.  For any sheaf
$\cG \in \csupp XC_{\ge v}$, we have $\Ext^k(\cF, \cG) = 0$ for all $k > w
- v$.
\end{prop}

We begin by proving a very special case of this result.

\begin{lem}\label{lem:ext-closed1}
Let $C \subset X$ be a closed orbit, and suppose $\cF \in \cg C$ is simple and $s$-pure of step $w$.  For any sheaf
$\cG \in \csupp XC_{\ge w}$, we have $\Ext^1(i_*\cF, \cG) = 0$.
\end{lem}
\begin{proof}
Consider the exact sequence
\[
\Hom(i_*\cF, \Sg {w+1}\cG) \to \Ext^1(i_*\cF, \Sl w\cG) \to
\Ext^1(i_*\cF,\cG) \to \Ext^1(i_*\cF, \Sg {w+1}\cG).
\]
The first term clearly vanishes, and the last term vanishes by Axiom~(S10) in the definition of an $s$-structure~\cite{a}.  Thus, the middle two terms are isomorphic.  To prove that $\Ext^1(i_*\cF,\cG) = 0$, we may replace $\cG$ by $\Sl w\cG$, and assume without loss of generality that $\cG$ is $s$-pure of step $w$.

Now, to every sheaf $\cG \in \csupp XC_{\ge w}$, we associate an invariant $\ell(\cG)$, defined to be the smallest integer $n$ such that $\cI_C^n\cG = 0$.  (See the proof of~\cite[Proposition~4.1]{a} for details.)  In the exact sequence
\[
0 \to \cI_C\cG \to \cG \to \cG/\cI_C\cG \to 0,
\]
we have $\ell(\cI_C\cG) = \ell(\cG) - 1$ and $\ell(\cG/\cI\cG) = 1$.  Now, $\cgl Xw$ is a Serre subcategory of $\cg X$, and because $C$ is a single closed orbit, $\csupp XC_{\ge w}$ is as well, as shown in the proof of~\cite[Proposition~10.1]{a}.  Thus, $\cI\cG$ and $\cG/\cI\cG$ are both also objects of $\csupp XC_{\ge w}$ that are $s$-pure of step $w$.

Consider the exact sequence
\[
\Ext^1(i_*\cF, \cI_C\cG) \to \Ext^1(i_*\cF,\cG) \to \Ext^1(i_*\cF,\cG/\cI\cG).
\]
If the first and last terms are known to vanish, the middle one must vanish
as well.  Thus, by induction on $\ell(\cG)$, we can reduce to the case
where $\ell(\cG) = 1$, {\it i.e.}, $\cI_C\cG = 0$.  In that case, there
must be a sheaf $\cG' \in \cg C$, $s$-pure of step $w$, such that $\cG
\cong i_*\cG'$.

If $\Ext^1(i_*\cF,i_*\cG') \ne 0$, then there is a nonsplit short exact
sequence
\begin{equation}\label{eqn:nonsplit}
0 \to i_*\cG' \to \cH \to i_*\cF \to 0.
\end{equation}
Note that $\cH$ is necessarily also $s$-pure of step $w$.  If $\ell(\cH) =
1$, then $\cH \cong i_*\cH'$ for some $\cH' \in \cg C$, and the entire
short exact sequence is the push-forward of the short exact sequence
\[
0 \to \cG' \to \cH' \to \cF \to 0
\]
in $\cg C$.  But $\Ext^1(\cF,\cG') = 0$ because the $s$-structure is
split, so this sequence splits, as does the one in~\eqref{eqn:nonsplit}. 
Thus, $\Ext^1(i_*\cF, i_*\cG') = 0$.

On the other hand, if $\ell(\cH) > 1$, then $\cI_C\cH \ne 0$.  Since
$\cI_C(i_*\cF) = 0$, $\cI_C\cH$ must be contained in the kernel of the map
$\cH \to i_*\cF$, so $\cI_C\cH$ can be identified with a subsheaf of
$i_*\cG'$.  That also implies that $i_*i^*\cI_C\cH \cong \cI_C\cH$.  Now,
$i^*\cI_C\cH$ is a quotient of $i^*\cI_C \otimes i^*\cH$.  Since $i^*\cI_C
\in \cgl C{-1}$ by assumption, and $i^*\cH \in \cgl Cw$, we conclude that
$\cI_C\cH \in \cgl X{w-1}$.  But that is a contradiction: $i_*\cG'$ is
$s$-pure of step $w$ and contains no nonzero subsheaf in $\cgl X{w-1}$. 
\end{proof}

To prove Proposition~\ref{prop:ext-closed}, we will carry out an Ext-group calculation using certain injective resolutions in the category of quasicoherent sheaves.  Let $\qg X$ denote the category of $G$-equivariant quasicoherent sheaves, and for any closed set $Z \subset X$, let
\[
\qsupp XZ_{\ge w} = \left\{ \cF \in \qg X \,\Big|
\begin{array}{c}
\text{every coherent subsheaf of}\\
\text{$\cF$ is in $\csupp XZ_{\ge w}$}
\end{array} \right\}.
\]

\begin{prop}\label{prop:inj-resoln}
Let $C \subset X$ be a closed orbit.  Every sheaf $\cF \in \qsupp XC_{\ge
w}$ admits an injective resolution
\[
 0 \to \cF \to \cI^0 \to \cI^1 \to \cdots
\]
with $\cI^k \in \qsupp XC_{\ge w+k}$.
\end{prop}
\begin{proof}
For brevity of notation, it will be convenient to set $\cI^{-1} = \cF$.
According to the proof of~\cite[Proposition~10.1]{a}, every sheaf in $\qsupp XC_{\ge w}$ has an injective hull in $\qsupp XC_{\ge w}$.  Let $\cI^0$ be such an injective hull of $\cF$, and let
$\partial^{-1}: \cF \to \cI^0$ be the inclusion map.  For subsequent terms
of the injective resolution, we proceed by induction.  Suppose that the
terms $\cI^{-1}, \cI^0, \ldots, \cI^n$ have already been constructed,
together with morphisms $\partial^k: \cI^k \to \cI^{k+1}$ for $k = -1,
\ldots, n-1$. We will show below that the cokernel of $\partial^{n-1}$ lies
in $\qsupp XC_{\ge w+n+1}$. Then, using the result from~\cite[Proposition~10.1]{a} again, we may take
$\cI^{n+1}$ to be an injective hull of $\cok \partial^{n-1}$ that also
lies in $\qsupp XC_{\ge w+n+1}$.

Suppose $\cok \partial^{n-1} \notin \qsupp XC_{\ge w+n+1}$.  Then there is
some coherent subsheaf $\cG \subset \cok \partial^{n-1}$ that does not lie
in $\csupp XC_{\ge w+n+1}$.  Replacing $\cG$ by its subsheaf $\Sl{w+n}\cG$,
we may assume that $\cG \in \cgl X{w+n}$.  Next, by replacing $\cG$ its
subsheaf $i_*i^!\cG$, we may assume that $\cG$ is actually supported
scheme-theoretically on the orbit $C$.  That is, $\cG \cong i_*\cG'$ for
some $\cG' \in \cgl C{w+n}$.  Finally, recall that $\cg C$ is a
finite-length category, so we may replace $\cG'$ by a simple subobject.  To
summarize: we have a coherent sheaf $\cG \subset \cok \partial^{n-1}$ such
that $\cG \cong i_*\cG'$ for some simple object $\cG' \in \cgl C{w+n}$.

Now, consider the preimage $\tilde\cG$ of $\cG$ in $\cI^n$.  Let $\cH
\subset \tilde\cG$ be any coherent subsheaf not contained in $\im
\partial^{n-1}$.  (Since $\tilde\cG$ is the union of all its coherent
subsheaves, such a sheaf $\cH$ exists.)  The map $\cH \to \cG$ is
surjective, because it is nonzero and $\cG$ is simple.  We thus have a
short exact sequence
\[
0 \to \cH \cap \im \partial^{n-1} \to \cH \to \cG \to 0.
\]
Now, by assumption, $\cI^n$ is the injective hull of $\im \partial^{n-1}$,
so $\cH$ cannot contain a direct summand complementary to $\cH \cap \im
\partial^{n-1}$.  In other words, the exact sequence above cannot split. 
We thus have $\Ext^1(i_*\cG', \cH \cap \im \partial^{n-1}) \ne 0$.  But since $\cH \cap \im \partial^{n-1} \in \csupp XC_{\ge x+n}$, this contradicts Lemma~\ref{lem:ext-closed1}.
\end{proof}

We are now ready to prove the main result of this section.

\begin{proof}[Proof of Proposition~\ref{prop:ext-closed}]
Let $\cI^*$ be an
injective resolution of $\cG$ as constructed in the previous proposition,
with
$\cI^k \in \qsupp XC_{\ge v+k}$.  In particular, if $k > w - v$, there are
no nonzero morphisms $\cF \to \cI^k$: the image of such a morphism, a
certain coherent subsheaf of $\cI^k$, belongs to $\cgl Xw$ and therefore
does not belong to $\csupp XC_{\ge v+k}$ unless it is $0$.  But any nonzero
element of $\Ext^k(\cF,\cG)$ can be represented
by a suitable nonzero morphism $\cF \to \cI^k$.
\end{proof}

We conclude with an application of this Ext-vanishing result.  The following technical lemma will be used in Section~\ref{sect:skew}.

\begin{lem}\label{lem:pullback-closed}
Let $i: Z \hto X$ be the inclusion of a closed subscheme, and let $t : C
\hto X$ be a closed orbit contained in $Z$, so that $i_C = i \circ t$.  Let
$\cF \in \cg X$ be such that $i_C^*\cF \in \cgl Cw$.  Then
$t^*h^{-r}(Li^*\cF) \in \cgl C{w-r}$ for all $r \ge 0$.
\end{lem}
\begin{proof}
We proceed by induction on $r$.  If $r = 0$, we have
$t^*h^0(Li^*\cF) \cong t^*i^*\cF \cong i_C^*\cF$, and that lies in $\cgl
Cw$ by assumption.  Now, assume that $t^*h^{-k}(Li^*\cF) \in \cgl C{w-k}$
for all $k < r$.  If $t^*h^{-r}(Li^*\cF) \notin \cgl C{w-r}$, then
there is some object $\cG \in \cgg C{w-r+1}$ such that
$\Hom(t^*h^{-r}(Li^*\cF), \cG) \ne 0$, or, equivalently,
$\Hom(h^{-r}(Li^*\cF), t_*\cG) \ne 0$.

Note that for $k < r$, the fact that $t^*h^{-k}(Li^*\cF) \in \cgl C{w-k}$
implies, by Proposition~\ref{prop:ext-closed}, that $\Hom(h^{-k}(Li^*\cF),
t_*\cG[n])
= 0$ whenever $n \ge r - k$.  Equivalently, we have 
$\Hom(h^{-k}(Li^*\cF)[k], t_*\cG[n]) = 0$ for all $n \ge r$. Since the object
$\Tg {-r+1}Li^*\cF \in \Db Z$ is built up by extensions from the objects
$h^{-k}(Li^*\cF)[k]$ with $k = 0, \ldots, r-1$, it follows that $\Hom(\Tg
{-r+1}Li^*\cF, t_*\cG[n]) = 0$ whenever $n \ge r$.

Next, from the distinguished triangle
\[
h^{-r}\cF[r] \to \Tg {-r}\cF \to \Tg {-r+1}\cF \to
\]
we obtain the long exact sequence
\begin{multline*}
\Hom(\Tg {-r+1}Li^*\cF, t_*\cG[r]) \to
\Hom(\Tg {-r}Li^*\cF, t_*\cG[r]) \to \\
\Hom(h^{-r}(Li^*\cF)[r], t_*\cG[r]) \to
\Hom(\Tg {-r+1}Li^*\cF, t_*\cG[r+1]) \to.
\end{multline*}
The first and last terms vanish by the preceding paragraph.  We saw
earlier that the third term is nonzero, so the second term is as well. 
The chain of isomorphisms
\[
\Hom(\cF,i_{C*}\cG[r]) \cong \Hom(Li^*\cF,t_*\cG[r]) \cong \Hom(\Tg
{-r}Li^*\cF, t_*\cG[r])
\]
shows then that $\Hom(\cF, i_{C*}\cG[r]) \ne 0$.  But this is a
contradiction: since $i_C^*\cF \in \cgl Cw$ and $\cG \in \cgg C{w-r+1}$,
we have $\Hom(\cF, i_{C*}\cG[r]) = 0$ by Proposition~\ref{prop:ext-closed}.
\end{proof}

\section{The Skew Co-$t$-structure}
\label{sect:skew}

Co-$t$-structures on triangulated categories have appeared in the work of Bondarko~\cite{bon} and Pauksztello~\cite{pauk}.  In this section, we construct a certain family of co-$t$-structures on $\Db X$, and we use them to define the notion of \emph{skew-purity}.

We begin by recalling the definition.
Given a triangulated category $\fD$ and a pair of full subcategories $(\fD_{\sqsubseteq 0}, \fD_{\sqsupseteq 0})$, let us set $\fD_{\sqsubseteq n} = \fD_{\sqsubseteq 0}[n]$ and $\fD_{\sqsupseteq n} = \fD_{\sqsupseteq 0}[n]$.  Note that this is the opposite of the usual convention with $t$-structures.  The pair $(\fD_{\sqsubseteq 0}, \fD_{\sqsupseteq 0})$ is called a \emph{co-$t$-structure} if the following three conditions hold:
\begin{enumerate}
\setcounter{enumi}{-1}
\item $\fD_{\sqsubseteq 0}$ and $\fD_{\sqsupseteq 0}$ are closed under direct summands.
\item $\fD_{\sqsubseteq 0} \subset \fD_{\sqsubseteq 1}$ and $\fD_{\sqsupseteq 0} \supset \fD_{\sqsupseteq 1}$.
\item $\Hom(A,B) = 0$ whenever $A \in \fD_{\sqsubseteq 0}$ and $B \in \fD_{\sqsupseteq 1}$.
\item For any object $X \in \fD$, there is a distinguished triangle $A \to X
\to B \to$ with $A \in \fD_{\sqsubseteq 0}$ and $B \in \fD_{\sqsupseteq 1}$.
\end{enumerate}
Note that for a co-$t$-structure, the distinguished triangle in Axiom~(3) is not functorial.  (The usual proof fails because $A \in \fD_{\sqsubseteq 0}$ does not imply $A[1] \in \fD_{\sqsubseteq 0}$.)  The properties of being \emph{bounded} or \emph{nondegenerate} are defined for co-$t$-structures in the same way as for $t$-structures.  The reader is referred to~\cite{bon} or~\cite{pauk} for further properties of co-$t$-structures.

Now, let $q: \orb X \to \Z$ be a function, to be known as a \emph{skew
perversity}.  Define a full subcategory of $\Dm X$ by
\[
\q\Dmkl Xw = \{ \cF \in \Dm X \mid 
\text{$h^k(X) \in \tql \cgl X{2w+2k}$ for all $k$} \}.
\]
Next, define a new function $\sdualq: \orb X \to \Z$, called the \emph{skew dual} of $q$, by
\[
\sdualq(C) = \alt C - \cod C - q(C).
\]
We then define a full subcategory of $\Dp X$ by
\[
\q\Dpkg Xw = \D(\breq\Dmkl X{-w}).
\]
As usual, we put
\[
\q\Dkl Xw = \q\Dmkl Xw \cap \Db X
\quad\text{and}\quad
\q\Dkg Xw = \q\Dpkg Xw \cap \Db X.
\]
The pictures of these categories resemble ``upside-down'' versions of the categories that constitute the staggered $t$-structure:
\[
\q\Dkl Xw:\incgr{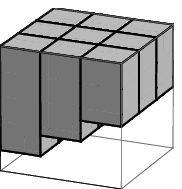}
\qquad\qquad
\q\Dkg Xw:\incgr{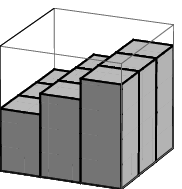}
\]
Finally, we define a full subcategory of $\Db X$ as follows:
\[
\q\Dkp Xw = \q\Dkl Xw \cap \q\Dkg Xw.
\]
Objects of $\q\Dkp Xw$ are said to be \emph{skew-pure} of \emph{skew-degree} $w$.

The following lemma collects some basic properties of these categories.  The proofs are routine and will be omitted.

\begin{lem}\label{lem:skew-easy}
\begin{enumerate}
\item $\q\Dmkl Xw$  and $\q\Dpkg Xw$ are closed under extensions and direct summands.
\item $\q\Dmkl Xw$ is stable under all standard truncation functors $\Tl
n$ and $\Tg n$.
\item $\q\Dmkl Xw[1] = \q\Dmkl X{w+1}$ and $\q\Dpkg Xw[1] = \q\Dpkg
X{w+1}$.
\item For every $\cF \in \Db X$, there exist integers $v, w$ such that $\cF \in \q\Dkg Xv \cap \q\Dkl Xw$. 
\qed
\end{enumerate}
\end{lem}

\begin{lem}\label{lem:spure-res}
Let $j: U \hto X$ be the inclusion of an open subscheme, and $i: Z \hto X$
the inclusion of a closed subscheme.  Then:
\begin{enumerate}
\item $j^*$ takes $\q\Dmkl Xw$ to $\q\Dmkl Uw$ and $\q\Dpkg Xw$ to
$\q\Dpkg Uw$.
\item $Li^*$ takes $\q\Dmkl Xw$ to $\q\Dmkl Zw$.
\item $Ri^!$ takes $\q\Dpkg Xw$ to $\q\Dpkg Zw$.
\item $i_*$ takes $\q\Dmkl Zw$ to $\q\Dmkl Xw$ and $\q\Dpkg Zw$ to
$\q\Dpkg Xw$.
\end{enumerate}
\end{lem}
\begin{proof}
For parts~(1) and~(4), the statements about $\q\Dmkl Xw$ follow from the
fact that $j^*$ and $i_*$ are exact, baryexact functors, and the
statements about $\q\Dpkg Xw$ then follow from the fact that $j^*$ and
$i_*$ commute with $\D$.  Part~(3) follows by duality from part~(2).

It remains only to prove part~(2).  Let $\cF \in \q\Dmkl Xw$.  We first
consider the special case where $\cF$ is concentrated in a single degree,
say degree $n$.  Thus, $\cF[n]$ is an object in $\tql\cgl X{2w+2n}$.  Let
$i_C : C \hto X$ be an orbit contained in $Z$, and let $j: U \hto X$ be the
inclusion of the open subscheme $U = X \ssm (\overline C \ssm C)$.  Thus,
$C$ is a closed orbit in $U$.  Let $t: C \hto Z \cap U$ be the inclusion of
$C$ into $Z \cap U$.  By assumption, $i_C^*\cF[n]|_C \in \cgl C{w
+ q(C) +n}$, so by Lemma~\ref{lem:pullback-closed},
$t^*h^{-r}(Li^*\cF[n]|_U) \in \cgl C{w + q(C) +n - r}$ for all $r \ge
0$.  Clearly, $t^*h^{-r}(Li^*\cF[n]|_U) \cong i_C^*h^{n-r}(Li^*\cF)|_C$,
so we have just shown that $h^k(Li^*\cF) \in \tql\cgl Z{2w+2k}$. 
Thus, $Li^*\cF \in \q\Dmkl Zw$.

Since $\q\Dmkl Zw$ is stable under extensions, an induction argument on the
number of nonzero cohomology sheaves shows that for all $\cF \in \q\Dkl
Xw$, we have $Li^*\cF \in \q\Dmkl Zw$.  Finally, consider a general object
$\cF \in \q\Dmkl Xw$.  For any $k$, we can form a distinguished triangle
\[
Li^*(\Tl {k-1}\cF) \to Li^*\cF \to Li^*(\Tg k\cF) \to
\]
Clearly, $\Tg k\cF \in \q\Dkl Xw$, so we already know that $Li^*(\Tg k\cF)
\in \q\Dmkl Zw$.  Moreover, the long exact sequence associated to the
distinguished triangle above shows that $h^k(Li^*\cF) \cong h^k(Li^*\Tg
k\cF)$, and hence that $h^k(Li^*\cF) \in \tql\cgl Z{2w+2k}$.  Since this
holds for all $k$, we conclude that $Li^*\cF \in \q\Dmkl Zw$, as desired.
\end{proof}

\begin{prop}\label{prop:skew-orth}
If $\cF \in \q\Dmkl Xw$ and $\cG \in \q\Dpkg X{w+1}$, then $\Hom(\cF,\cG) = 0$.
\end{prop}
\begin{proof}
We proceed by noetherian induction, and assume the result is already known
for all proper closed subschemes of $X$.  Let $a$ be an integer such that
$\cG \in \Dpg Xa$.  Then $\Hom(\cF, \cG) \cong \Hom(\Tg a\cF,\cG)$. 
Moreover, we have $\Tg a\cF \in \q\Dkl Xw$.  Thus, we may reduce to the
case where $\cF$ actually belongs to $\q\Dkl Xw$, by replacing $\cF$ by
$\Tg a \cF$ if necessary.  Next, recall that $\Hom(\cF,\cG) \cong
\Hom(\D\cG, \D\cF)$, and suppose $\D\cF \in \Dg Xb$.  We may similarly
reduce to the case where $\cG \in \q\Dkg X{w+1}$ by replacing $\cG$ by
$\D\Tg b\D\cG$ if necessary.

Once we have reduced to the case where both $\cF$ and $\cG$ are bounded, we
may, by induction on the number of nonzero cohomology sheaves, further
reduce to the case where $\cF$ and $\D\cG$ are each concentrated in a
single degree.  Suppose that $\cF$ is concentrated in degree $k$, and
$\D\cG$ in degree $m$.  That is, $\cF[k] \in \tql\cgl X{2w+2k}$, and
$(\D\cG)[m] \in \tbql \cgl X{-2w-2+2m}$.

Let $C \subset X$ be an open orbit, and let $U \subset X$ be the
corresponding (possibly nonreduced) subscheme.  Consider the usual exact
sequence
\[
\lim_{\substack{\to \\ Z'}} \Hom(Li^*_{Z'}\cF, Ri^!_{Z'}\cG) \to
\Hom(\cF,\cG) \to 
\Hom(\cF|_U,\cG|_U) 
\]
where $i_{Z'}: Z' \hto X$ ranges over all closed subscheme structures on
$X \ssm U$.  Since $Li^*_{Z'}\cF \in \q\Dmkl {Z'}w$ and $Ri^!_{Z'}\cG \in
\q\Dpkg X{w+1}$, the first term vanishes by assumption.  To finish the
proof, then, it suffices to show that the third term vanishes.

Since the associated reduced scheme of $U$ is the single orbit $C$, $U$
has no nonempty ($G$-invariant) proper open subschemes.  The fact that
$\D\cG|_U$ is concentrated in degree $m$ then implies, by \cite[Lemma~6.6]{a}, that
$\cG|_U$ is concentrated in degree $\cod C - m$.  Since
\[
\D\cG[m]|_U \in \tbql\cgl U{-2w-2+2m} = \cgl U{\alt C - \cod C -
q(C)-w-1+m},
\]
we know that $\cG[\cod C -m]|_U \in \cgg U{\cod
C + q(C)+w+1-m}$, and hence that $\cG[\cod C - m]|_U \in \csupp UC_{\ge \cod C +
q(C)+w+1-m}$. Similarly, 
\[
\cF[k]|_U \in \tql\cgl U{2w+2k} = \cgl
U{q(C)+w+k},
\]
and therefore $i_C^*\cF[k]|_C \in \cgl C{q(C)+w+k}$. 
Propostion~\ref{prop:ext-closed} tells us that 
\[
\Hom(\cF[k]|_U, \cG[\cod
C-m+n]|_U) = 0
\]
whenever $n > k + m - \cod C -1$.  In particular, taking
$n = k + m - \cod C$, we find that $\Hom(\cF[k]|_U, \cG[k]|_U) \cong
\Hom(\cF|_U,\cG|_U) = 0$, as desired.
\end{proof}

\begin{thm}
$(\q\Dkl X0, \q\Dkg X0)$ is a nondegenerate, bounded co-$t$-structure on $\Db X$.
\end{thm}
\begin{proof}
We proceed by noetherian induction, in a manner similar to the proof of Proposition~\ref{prop:purep-t}.  In view of the results above, it remains only to show that for any $\cF \in \Db X$, there is a distinguished triangle $\cF' \to \cF \to \cF'' \to$ with $\cF'
\in \q\Dkl X0$ and $\cF'' \in \q\Dkg X1$.  Let us first treat the special case where $\cF$ is concentrated in a single degree, say $\cF \cong h^k(\cF)[-k]$.  
Choose an open orbit $C \in \orb X$ on which $\cod C$ achieves its minimum value, and let $U \subset X$ be the corresponding open subscheme.  Consider the sheaf $\Sl {q(C)+k}(\cF[k])|_U \in \cgl U{q(C)+k} = \tql\cgl U{2k}$.  By~\cite[\lemqext]{at:bs}, there exists a subsheaf of $\cF[k]$ in $\tql\cgl X{2k}$ whose restriction to $U$ is $\Sl {q(C)+k}(\cF[k])|_U$.  Denote this subsheaf by $\cF_1[k]$.  That is, we denote by $\cF_1$ an object of $\Db X$ concentrated in degree $k$ such that $\cF_1[k]$ is the subsheaf of $\cF[k]$ obtained by invoking \cite[\lemqext]{at:bs}.  Clearly, $\cF_1 \in \q\Dkl X0$.

Next, let $\cF'$ be the cone of the obvious morphism $\cF_1 \to \cF$.  Clearly,  $\cF'$ is also concentrated in degree $k$, and $\cF'[k]|_U \cong \Sg {q(C)+k+1}(\cF[k]|_U)$.  Because $C$ was chosen to minimize $\cod C$, we have $\D\cF' \in \Dg X{\cod C - k}$.  Moreover, by~\cite[Proposition~6.8]{a}, $\D\cF'|_U$ is concentrated in degree $\cod C - k$, and $(\D\cF')[\cod C - k]|_U \in \cgl U{\alt C - q(C) -k -1} = \cgl U{\sdualq(C) + \cod C - k - 1} = \tbql\cgl U{2(\cod C -k) - 2}$.  By invoking~\cite[\lemqext]{at:bs} again, we can find an object $\cG_1 \in \Db X$, concentrated in degree $\cod C - k$, such that $\cG_1[\cod C - k]$ is a subsheaf of $(\D\cF')[\cod C - k]$ lying in $\tbql\cgl X{2(\cod C - k) - 2}$, and such that $\cG_1|_U \cong \D\cF'|_U$.  By construction, $\cG_1 \in \breq\Dkl X{-1}$.

Let $\cF_2 = \D\cG_1$, and let $\cG$ denote the cocone of the morphism $\cF' \to \cF_2$.  We have
\[
\cF \in \{\cF_1\} * \{\cG\} * \{\cF_2\},
\]
with $\cF_1 \in \q\Dkl X0$ and $\cF_2 \in \q\Dkg X1$.  Moreover, since $\cF'|_U \cong \cF_2|_U$, we see that $\cG$ is supported on a proper closed subscheme.  It follows by noetherian induction that $\cF$ sits in a suitable distinguished triangle.

Now, for general $\cF \in \Db X$, we proceed by induction on the number of nonzero cohomology sheaves.  Choose some $k$ such that $\Tl k\cF$ and $\Tg{k+1}\cF$ are both nonzero, and thus have fewer nonzero cohomology sheaves than $\cF$.  Find distinguished triangles
\[
\cF'_1 \to \Tl k\cF \to \cF''_1 \to
\qquad\text{and}\qquad
\cF'_2 \to \Tg{k+1}\cF \to \cF''_2 \to
\]
with $\cF'_1,\cF'_2 \in \q\Dkl X0$ and $\cF''_1,\cF''_2 \in \q\Dkg X1$.  Consider the composition $\cF'_2[-1] \to \Tg{k+1}[-1]\cF \to \Tl k\cF$, which we denote by $f$.  Now, $\Hom(\cF'_2[-1],\cF''_1) = 0$ (because $\cF'_2[-1] \in \q\Dkl X{-1}$), so $f \in \Hom(\cF'_2[-1],\Tl k\cF)$ factors through $\cF'_1$.  We thus obtain a commutative square
\[
\xymatrix@=10pt{
\cF'_2[-1] \ar[r]\ar[d] & \Tg{k+1}\cF[-1] \ar[d] \\
\cF'_1 \ar[r] & \Tl k\cF}
\]
Let us complete this diagram using the $9$-lemma~\cite[Proposition~1.1.11]{bbd}, and then rotate:
\[
\xymatrix@=10pt{
\cF'_1 \ar[r]\ar[d] & \Tl k\cF \ar[r]\ar[d] & \cF''_1 \ar[r]\ar[d]
& \\
\cF' \ar[r]\ar[d] & \cF \ar[r]\ar[d] & \cF'' \ar[r]\ar[d] & \\
\cF'_2 \ar[r]\ar[d] & \Tg {k+1}\cF \ar[r]\ar[d] & \cF''_2
\ar[r]\ar[d] & \\
&&&}
\]
We see that $\cF' \in \q\Dkl X0$ and $\cF'' \in \q\Dkg X1$ because those categories are stable under extensions.
\end{proof}

\section{The Skew Purity Theorem}
\label{sect:purity2}

We prove the skew version of the Purity Theorem in this section.  Of course, we must specify a skew perversity with respect to which skew-purity statements are to be understood.  Given a moderate staggered perversity $r: \orb X \to \Z$, we associate to it a skew perversity, denoted $\skewr: \orb X \to \Z$, as follows:
\[
\skewr(C) = r(C) - \cod C.
\]
Note that this operation transforms staggered duals into skew duals:
\[
\skewdr(C) = (\scod C - r(C)) - \cod C = \alt C - \cod C - (r(C) - \cod C) = (\skewr)\breve{\;}(C).
\]
Henceforth,we will generally omit the perversity from the notation for skew categories.  Unless otherwise specified, the categories $\Dmkl Xw$ and $\Dpkg Xw$ should be understood to be defined with respect to $\skewr(C)$.

\begin{lem}\label{lem:ic-pure2}
Let $\cL \in \cg {C_0}$ be a coherent sheaf, $s$-pure of step $v$.  For any staggered perversity $r$, the object $\ru\cIC(\overline C_0, \cL[v-r(C_0)])$
is skew-pure of skew degree $w = 2v - 2r(C_0) +  \cod C_0$.
\end{lem}
\begin{proof}
Let $j: C_0 \hto \overline C_0$ be the inclusion, and let $\cF = \ru
j_{!*}(\cL[v-r(C_0)]$.  Of course, $\ru\cIC(\overline C_0,
\cL[v-r(C_0)]) \cong i_{C_0*}\cF$, so it suffices to show that $\cF$ is
skew-pure of skew degree $w$.

We saw in the proof of Lemma 5.1 that $h^k(\cF) = 0$ for $k < r(C_0) - v$.
Next, let $u = 2v - \alt C_0$, and consider the function $q: \orb {\overline C_0} \to \Z$ given by
\[
q(C) = 2\skewr(C) + 2w + 2(r(C_0) -v)  - u.
\]
Direct calculation shows that $q(C)$ satisfies the condition of
Corollary~\ref{cor:ic-bound}.  That statement tells us that $\cF \in
\q\Dsl {\overline C_0}{u}$, or, equivalently, that
\[
\cF \in \tsr\Dsl
{\overline C_0}{2w + 2(r(C_0)-v)}.
\]
In other words, for all $k \ge r(C_0) - v$, we have
\[
h^k(\cF) \in \tsr\cgl {\overline C_0}{2w+2(r(C_0)-v)} \subset \tsr\cgl
{\overline C_0}{2w + 2k}.
\]
Thus, $\cF \in \lsr\Dkl {\overline C_0}w$.  The same argument shows that
$\D\cF \cong \bru j_{!*}(\D(\cL[v- r(C_0)]))$ belongs to $\lsdr\Dkl
{\overline C_0}{w'}$, where $w' = 2(\alt C_0 - v) - 2\dualr(C_0) + \cod
C_0 = -w$.   Thus, $\cF \in \Dkp {\overline C_0}w$, as desired.
\end{proof}

\begin{thm}[Skew Purity]\label{thm:purity2}
Suppose $X$ is endowed with a recessed, split $s$-structure.
Let $r: \orb X \to \Z$ be a staggered perversity.
\begin{enumerate}
\item Let $\cF$ be a staggered
sheaf.  If $\cF \in \Dkl Xw$, then every
subquotient of $\cF$ is in $\Dkl Xw$.  If $\cF \in \Dkg Xw$, then every
subquotient of $\cF$ is in $\Dkg Xw$.
\item  Every simple staggered sheaf
is skew-pure.
\item  Every staggered sheaf
$\cF$ admits a unique finite filtration
\[
\cdots \subset \cF_{\sqsubseteq w-1} \subset \cF_{\sqsubseteq w} \subset \cF_{\sqsubseteq w+1}
\subset \cdots
\]
such that $\cF_{\sqsubseteq w}/\cF_{\sqsubseteq w-1}$ is skew-pure of skew degree $w$.
\item Let $\cF \in \Db X$.  Then
$\cF \in \Dkl Xw$ if and only if $\ru
h^i(\cF) \in \Dkl X{w+i}$ for all $i$, and $\cF \in \Dkg Xw$ if and only if
$\ru h^i(\cF) \in \Dkg X{w+i}$ for all $i$.
\end{enumerate}
\end{thm}
\begin{proof}
(1)~We will prove the statement for $\Dkl Xw$; the statement for $\Dkg Xw$
then follows by duality.  Note that any subquotient of $\cF$ arises by
extensions among the composition factors of $\cF$, so it suffices to prove
that every composition factor of $\cF$ is in $\Dkl Xw$.  If $\cF$
is simple, then it is skew-pure by Lemma~\ref{lem:ic-pure2}, and there is
nothing to prove.  Otherwise, let $\cF_1$ be a simple quotient of $\cF$. 
Since $\Hom(\cF,\cF_1) \ne 0$, we know by Proposition~\ref{prop:skew-orth}
that $\cF_1 \notin \Dkg X{w+1}$.  Since $\cF_1$ is skew-pure, it must lie in
$\Dkl Xw$.  Therefore, $\cF_1[-1] \in \Dkl Xw$ as well.  Let $\cF_2 \subset
\cF$ be the kernel of the morphism $\cF \to \cF_1$.  From the
distinguished triangle $\cF_1[-1] \to \cF_2 \to \cF \to $, we see that
$\cF_2 \in \Dkl Xw$.  Since $\cF_2$ has shorter length than $\cF$, we know
that all it composition factors lie in $\Dkl Xw$.  Thus, all composition
factors of $\cF$ lie in $\Dkl Xw$, as desired.

(2)~This was proved in Lemma~\ref{lem:ic-pure2}.

(3)~We follow the proof of~\cite[Th\'eor\`eme~5.3.5]{bbd}.  Given an
integer $w$, let $\cS^+$ (resp.~$\cS^-$) denote the set of isomorphism
classes of simple staggered sheaves of skew-degree $> w$ (resp.~$\le w$). 
Clearly, if $\cG \in \cS^-$ and $\cG' \in \cS^+$, then $\cG'[1] \in \Dkg
X{w+1}$ as well, so $\Hom(\cG, \cG'[1]) = 0$ by
Proposition~\ref{prop:skew-orth}.  The sets $\cS^+$ and $\cS^-$ thus
satisfy the hypotheses of~\cite[Lemme~5.3.6]{bbd}, which then tells us
that every staggered sheaf $\cF$ admits a unique subobject $\cF_{\sqsubseteq w}$
belonging to $\Dkl Xw$ such that the quotient $\cF/\cF_{\sqsubseteq w}$ belongs to
$\Dkg X{w+1}$.  The functoriality of this assignment guarantees that
$\cF_{\sqsubseteq w-1} \subset \cF_{\sqsubseteq w}$ (so that we do indeed obtain a
filtration) and that $\cF_{\sqsubseteq w}/\cF_{\sqsubseteq w-1}$ is skew-pure of skew
degree $w$.  Finally, the uniqueness of this filtration follows from
part~(1).

(4)~Again, we will prove only the statement about $\Dkl Xw$.  First,
suppose $\ru h^i(\cF) \in \Dkl X{w+i}$ for all $i$.  Then $\ru
h^i(\cF)[-i] \in \Dkl Xw$.  Using the fact that $\Dkl Xw$ is stable under
extensions, it follows by induction on the number of nonzero $\ru
h^i(\cF)$ that $\cF \in \Dkl Xw$ as well.  Conversely, suppose $\cF \in
\Dkl Xw$.  By a minor abuse of terminology, we define the
\emph{total length} of $\cF$ to be the sum of lengths of all $\ru
h^i(\cF)$.  We proceed by induction on total length.  Let $k$ be the
largest integer such that $\ru h^k(\cF) \ne 0$, and let $\cF_1$ be a
simple quotient of $\ru h^k(\cF)$.  Note that $\ru\Tg k\cF \cong \ru
h^k(\cF)[-k]$.  From the adjunction $\Hom(\cF,\cF_1[-k]) \cong
\Hom(\ru\Tg k\cF, \cF_1[-k])$, we see that there is a natural nonzero
morphism $\cF \to \cF_1[-k]$.  $\cF_1[-k]$ is skew-pure and not in $\Dkg
X{w+1}$, by Proposition~\ref{prop:skew-orth}, so $\cF_1[-k] \in \Dkl Xw$,
or $\cF_1 \in \Dkl X{w+k}$.  Next, let $\cF_2$ be the cocone of the
morphism $\cF \to \cF_1[-k]$.  From the distinguished triangle
\[
\cF_1[-k-1] \to \cF_2 \to \cF \to
\]
and the fact that $\cF_1[-k-1] \in \Dkl Xw$, we see that $\cF_2 \in \Dkl
Xw$ as well.  It has shorter total length, so by assumption, $\ru
h^i(\cF_2) \in \Dkl X{w+i}$ for all $i$.  Now, consider the cohomology long
exact sequence associated to the distinguished triangle above.  We see
that $\ru h^i(\cF) \cong \ru h^i(\cF_2)$ for $i < k$, whereas for $i =
k$, we have a short exact sequence $0 \to \ru h^k(\cF_2) \to \ru h^k(\cF)
\to \cF_1 \to 0$.  It follows that $\ru h^i(\cF) \in \Dkl X{w+i}$ for all
$i$, as desired.
\end{proof}

\section{The Decomposition Theorems}
\label{sect:decomp}

In this section, we prove the two versions of the Decomposition Theorem. 
In contrast with the two Purity Theorems, whose proofs involved different
arguments, the two Decomposition Theorems have essentially identical
proofs, and we will prove them simultaneously.

We retain the assumption that $X$ is endowed with a recessed, split
$s$-structure.  Let $r: \orb X \to \Z$ be a fixed staggered perversity. 

\begin{prop}\label{prop:degrees}
Let $\cF \in \ru\cM(X)$.  The following conditions are equivalent:
\begin{enumerate}
\item $\cF$ is simple and skew-pure of skew degree $w$.

\item $\cF \cong \cIC(\overline C, \cL[(w - \cod C)/2])$, where $\cL \in
\cg C$ is an irreducible vector bundle that is $s$-pure of step $(w - \cod
C)/2 + r(C)$
\end{enumerate}
If furthermore $r$ is moderate, these conditions are equivalent to
\begin{enumerate}
\item[(3)] $\cF$ is simple and pure of baric degree $w + r(c) - \dualr(C)$.
\end{enumerate}

In particular, in the case where $r(C) = \half\scod C$, the baric and skew
degrees of a simple staggered sheaf coincide. 
\end{prop}
\begin{proof}
We know that every simple staggered sheaf is of the form
$\ru\cIC(\overline C, \cL[v - r(C)])$ for some irreducible vector bundle
$\cL$.  From Lemma~\ref{lem:ic-pure2}, we have $v - r(C) = (w - \cod
C)/2$, and this establishes the equivalence of parts~(1) and~(2).  Next, in case $r$ is moderate,
Lemma~\ref{lem:ic-pure} tells us that the baric degree of $\cF$ is
\[
2v - \alt C = w - \cod C + 2r(C) - \alt C = w + r(C) + (r(C) - \scod C)
= w + r(C) - \dualr(C).
\]
This establishes the equivalence of part~(3) with the other two.
\end{proof}

Note that in the special case of the self-dual staggered perversity $r(C) =
\half\scod C$, the baric degree and skew degree of a simple staggered
sheaf coincide.

\begin{prop}\label{prop:pure-ext}
Let $\cF$ and $\cG$ be staggered sheaves.
\begin{enumerate}
\item If $\cF$ is skew-pure of skew degree $w$ and $\cG$ is skew-pure of
skew degree $v$, then $\Hom(\cF, \cG[k]) = 0$ for all $k > w - v$.
\item Assume that $r(C) = \half\scod C$ and $r$ is moderate.  If $\cF$ is pure of baric
degree $w$ and $\cG$ is pure of
baric degree $v$, then $\Hom(\cF, \cG[k]) = 0$ for all $k > w - v$.
\end{enumerate}
\end{prop}
\begin{proof}
Part~(1) is an immediate consequence of Proposition~\ref{prop:skew-orth}
and the fact that $\cG[k] \in \Dkg X{v+k}$, and part~(2) then follows using
Proposition~\ref{prop:degrees}.
\end{proof}

\begin{prop}[{{\it cf.}~\cite[Th\'eor\`eme~5.3.8]{bbd}}]
\ 
\begin{enumerate}
\item Every skew-pure staggered sheaf is semisimple.
\item Assume $r(C) = \half\scod C$ and $r$ is moderate.  Then every pure staggered sheaf is
semisimple.
\end{enumerate}
\end{prop}
\begin{proof}
The proofs of the two parts are identical, and we prove them
simultaneously. Let $\cF$ be a (skew-)pure staggered sheaf, and let $\cF'
\subset \cF$ be the sum
of all simple subobjects of $\cF$.  $\cF'$ is the largest semisimple
subobject of $\cF$.  We must show that $\cF' = \cF$.  Form a short exact
sequence
\[
0 \to \cF' \to \cF \to \cF'' \to 0.
\]
By Theorem~\ref{thm:purity} or~\ref{thm:purity2}, $\cF'$ and $\cF''$ are
also (skew-)pure of degree $w$, and then by
Proposition~\ref{prop:pure-ext}, $\Hom(\cF'', \cF'[1])
= 0$.  It follows that this short exact sequence splits, and that $\cF
\cong \cF' \oplus \cF''$.  If $\cF'' \ne 0$, then any simple subobject of
$\cF''$ would also be a simple subobject of $\cF$ not contained in $\cF'$,
a contradiction.
\end{proof}

\begin{prop}[{{\it cf.}~\cite[Th\'eor\`eme~5.4.5]{bbd}}]
Let $\cF \in \Db X$.
\begin{enumerate}
\item If $\cF$ is skew-pure, then $\cF \cong \bigoplus_{i \in \Z}
\ru
h^i(\cF)[-i]$.
\item Assume $r(C) = \half\scod C$ and that $r$ is moderate.  If $\cF$ is pure, then
$\cF \cong \bigoplus_{i \in \Z} \ru
h^i(\cF)[-i]$.
\end{enumerate}
\end{prop}
\begin{proof}
Again, we prove the two parts simultaneously.  We proceed by induction on
the number of nonzero staggered cohomology objects of
$\cF$.  If $\cF$ has zero or one nonzero cohomology objects, then there is
nothing to prove.  Otherwise, let $k$ be the largest integer such that $\ru
h^k(\cF) \ne 0$, and form the distinguished triangle
\[
\ru\Tl {k-1}\cF \to \cF \to \ru h^k(\cF)[-k] \to
\]
It follows from Theorem~\ref{thm:purity} or~\ref{thm:purity2} that the
staggered truncation functor $\ru\Tl {k-1}$ preserves (skew-)purity.  Since
$\ru\Tl {k-1}\cF$ has fewer nonzero cohomology objects than $\cF$, we have
$\ru\Tl {k-1}\cF \cong \bigoplus_{i \le k-1} \ru h^i(\cF)[-i]$ by
assumption.  Then
\begin{align*}
\Hom(\ru h^k(\cF)[-k], (\ru\Tl {k-1}\cF)[1])
&\cong\bigoplus_{i \le k-1} 
\Hom(\ru h^k(\cF)[-k], \ru h^i(\cF)[-i+1]) \\
&\cong\bigoplus_{i \le k-1} 
\Hom(\ru h^k(\cF), \ru h^i(\cF)[k+1-i]).
\end{align*}
We claim that $\Hom(\ru h^k(\cF), \ru h^i(\cF)[k+1-i]) = 0$ for all $i$. 
In the setting of skew-purity, Theorem~\ref{thm:purity2} tells us that
$\ru h^k(\cF)$ is skew-pure of skew degree $w+k$, and that each $\ru
h^i(\cF)$ is skew-pure of skew degree $w+i$.  In the setting of baric
purity, Theorem~\ref{thm:purity} tells us that $\ru h^k(\cF)$ and all the
$\ru h^i(\cF)$ are pure of baric degree $w$.  Since $k+1-i > (w+k) -
(w+i)$ and $k+1-i > 0$, Proposition~\ref{prop:pure-ext} tells us in both
cases that $\Hom(\ru h^k(\cF), \ru h^i(\cF)[k+1-i]) = 0$.  Thus, $\Hom(\ru
h^k(\cF)[-k],
(\ru\Tl {k-1}\cF)[1]) = 0$, so in the distinguished triangle above, we find
that
\[
\cF \cong \ru\Tl {k-1}\cF \oplus \ru h^k(\cF)[-k] \cong \bigoplus_{i \in
\Z} \ru h^i(\cF)[-i],
\]
as desired.
\end{proof}

Combining the preceding two propositions with the formulas in
Proposition~\ref{prop:degrees} relating step, baric degree, and skew
degree, we obtain the following theorem.

\begin{thm}[Decomposition]
\label{thm:decomposition}
Assume that $X$ is endowed with a recessed, split $s$-structure. 
\begin{enumerate}
\item Every skew-pure complex $\cF \in \Dkp Xw$ admits a decomposition
\[
\cF \cong \bigoplus_{i=1}^n \ru\cIC(\overline C_i, \cL_i[(w-k_i-\cod
C)/2])[k_i],
\]
where each $\cL_i \in \cg {C_i}$ is an irreducible vector bundle that is
$s$-pure of step $(w - k_i - \cod C)/2 + r(C_i)$.
\item Assume $r(C) = \half\scod C$ and that $r$ is moderate.  Every pure complex $\cF \in \Dsp
Xw$ admits a decomposition
\[
\cF \cong \bigoplus_{i=1}^n
\cIC(\overline C_i, \cL_i[(w -\cod C_i)/2])[k_i]
\]
where each $C_i$ is an orbit such that $w \equiv \cod C_i \pmod 2$, and
each $\cL_i \in \cg {C_i}$ is an irreducible vector bundle that is
$s$-pure of step $(w+\alt C_i)/2$.\qed
\end{enumerate}
\end{thm}

\section{An Example}
\label{sect:example}

We conclude with a brief example illustrating the skew decomposition
theorem.  Let $A =  \C[x,y,z]$, and let $X = \bbA^3(\C) = \Spec A$.  Let $G_1 = G_2 = G_3 = \bbG_m$, and let $G = G_1 \times G_2 \times G_3$.  (This notation will facilitate distinguishing between the various factors of $G$.) Let $G$ act on $X$ in the usual way: 
$(t_1,t_2,t_3) \cdot (a_1,a_2,a_3) \mapsto (t_1a_1, t_2a_2, t_3a_3)$
for $(t_1,t_2,t_3) \in G$ and $(a_1,a_2,a_3) \in \bbA^3$.

For each $i \in \{1,2,3\}$, let $X_i \cong \Z$ denote the character lattice of $G_i$.  For any subset $S \subset \{ 1,2,3\}$, let $G_S = \prod_{i \in S} G_i$.  Its character lattice is $X_S = \bigoplus_{i \in S} X_i$.  In this way, we regard each $X_S$ as a direct summand (rather than merely a quotient) of $X(G) = X_1 \oplus X_2 \oplus X_3$.  Now, let $\chi: X(G) \to \Z$ be the map $(\lambda_1, \lambda_2, \lambda_3) \mapsto \lambda_1+\lambda_2+\lambda_3$.  By restriction, $\chi$ gives rise to maps $\chi_S: X_S \to \Z$ for all $S \subset \{1,2,3\}$.  The $G$-stabilizer of any point is some $G_S$, so, following Section~\ref{sect:orbit} and the gluing theorem for $s$-structures~\cite[Theorem~1.1]{as:flag}, the collection $\{\chi_S\}$ defines an $s$-structure on $\bbA^3$.  Taking the dualizing complex $\omega_{\bbA^3}$ to be the structure sheaf, one may calculate that $\alt C = \cod C$ for every orbit $C$.

Throughout this example, we pass freely between the language of
$G$-equivariant coherent sheaves on $X$ and that of $A$-modules
with a compatible $G$-action.  For $\lambda \in X(G)$, let $A(\lambda)$ denote a rank-$1$ free $A$-module generated by an element
on which $G$ acts by $\lambda$. Let $\C(\lambda)$ denote the
$1$-dimensional $A$-module on which $x$, $y$, and $z$ act by $0$, and $G$
acts by $\lambda$.  More generally, for any coherent sheaf $\cF$, let
$\cF(\lambda)$ denote the sheaf $\cF \otimes A(\lambda)$.  
The object $\cH_\lambda = \C(\lambda)[\chi(\lambda) - 3]$ is a simple
staggered sheaf with respect to the middle perversity $r(C) = \half \scod C$.  Its baric and skew degrees are both $2\chi(\lambda) - 3$.

Consider the structure sheaves of the $x$- and $z$-axes:
\[
\cO_x = A/(y,z), \qquad
\cO_z = A/(x,y).
\]
We claim that $\cO_x \Lotimes \cO_z$ is skew-pure of skew degree $0$.  It is easy to see that $\cO_z \in \Dkl X0$.  Then, $\cO_x \Lotimes \cO_z
\cong i_*Li^*\cO_z$, where $i: \Spec \cO_x \hto X$ is the inclusion of the
$x$-axis as a reduced closed subscheme.  By Lemma~\ref{lem:spure-res}, we
know that $\cO_x \Lotimes \cO_z \in \Dkl X0$.  Now, consider the dual:
\[
\D(\cO_x \Lotimes \cO_z) \cong
\cRHom(\cO_x \Lotimes \cO_z, A) \cong
\cRHom(\cO_x, \cRHom(\cO_z,A)).
\]
Direct computation shows that $\cRHom(\cO_z,A) \cong
\cO_z(1,1,0)[-2]$.  To compute $\cRHom(\cO_x, \cO_z(1,1,0)[-2])$, we use
the following free resolution of $\cO_x$:
\[
yzA \to yA \oplus zA \to A
\qquad\text{or}\qquad
A(0,-1,-1) \to A(0,-1,0)\oplus A(0,0,-1) \to A.
\]
Then $\cRHom(\cO_x, \cO_z(1,1,0)[-2])$ is represented by the complex
\[
\cO_z(1,1,0) \to \cO_z(1,2,0) \oplus \cO_z(1,1,1) \to \cO_z(1,2,1),
\]
with nonzero terms in degrees $2$, $3$, and $4$.  The term in degree $k$ lies in $\cgl Xk$, so $\D(\cO_x \Lotimes \cO_z) \in \Dkl X0$, and $\cO_x \Lotimes \cO_z$ is skew-pure of skew degree $0$.

In fact, it turns out that
\[
\cO_x \Lotimes \cO_z \cong \C(0) \oplus \C(0,-1,0)[1]
\cong \cH_0[3] \oplus \cH_{(0,-1,0)}[5].
\]


\end{document}